\documentclass[11pt, letterpaper]{amsart}

\usepackage{amsmath}
\usepackage{tikz}
\usepackage{amssymb}

\usepackage{amsthm}
\usepackage{tikz-cd}
\usepackage[all]{xy}
\usepackage{amsfonts}

\DeclareMathAlphabet{\mathpzc}{OT1}{pzc}{m}{it}

\usepackage{mathrsfs}
\usepackage{hyperref}
\usepackage{xcolor}

\usepackage{mathtools}
\usepackage{subfig, caption} 
\usepackage{wrapfig}
\newcommand{\llangle}{\langle\langle}
\newcommand{\rrangle}{\rangle\rangle}
\newcommand{\T}{\mathcal{T}}
\newcommand{\C}{\mathcal{C}}
\newcommand{\la}{\lambda}

\newcommand{\N}{\mathbb{N}}
\newcommand{\bH}{\mathbb{H}}
\newcommand{\R}{\mathbb{R}}
\newcommand{\Z}{\mathbb{Z}}

\newcommand{\I}{\mathcal{I}}

\newcommand{\Mod}{\operatorname{Mod}}
\newcommand{\PMod}{\operatorname{PMod}}
\newcommand{\Homeo}{\operatorname{Homeo}}

\newcommand{\Aut}{\operatorname{Aut}}

\renewcommand{\H}{\mathcal{H}}

\makeatletter
\let\c@equation\c@subsection
\makeatother
\numberwithin{equation}{section} 

\usepackage{caption}

\newtheorem{theorem}{Theorem}[section]

\newtheorem*{claim*}{Claim}

\newtheorem{lemma}[theorem]{Lemma}

\theoremstyle{definition}
\newtheorem{definition}[theorem]{Definition}
\newtheorem{example}[theorem]{Example}

\newtheorem{question}[theorem]{Question}

\theoremstyle{remark}
\newtheorem{remark}[theorem]{Remark}

\title{Normal generators for mapping class groups}

\author{Hyungryul Baik}
\address{\hskip-\parindent
Department of Mathematical Sciences\\
Korea Advanced Institute of Science and Technology (KAIST)\\
291 Daehak-ro, Yuseong-gu,  Daejeon 34141, Republic of Korea}
\email{hrbaik@kaist.ac.kr }

\author{Dongryul M. Kim}
\address{\hskip-\parindent
Department of Mathematics\\
Yale University\\
219 Prospect St, New Haven, CT 06511, USA}
\email{dongryul.kim97@gmail.com}

\date{\today}

\usepackage{imakeidx}
\makeindex

\begin{document}

\begin{abstract}
In this chapter, we discuss normal generators for mapping class groups of surfaces. Especially, we focus on the relation between normal generation of a mapping class with its asymptotic translation lengths on the Teichm\"uller space and the curve graph of the underlying surface. We also discuss several open questions.
\end{abstract}


\maketitle
\tableofcontents

\section{Introduction}

Ever since Thurston brought it to prominence \cite{Thurston_construction}, the mapping class group of a surface has become a ubiquitous object in the study of geometry, topology, and dynamics in low-dimensional settings. He classified homeomorphisms of surfaces, a classification which divides mapping classes into three types: periodic, reducible, and pseudo-Anosov (\cite{Nielsen_classification}, \cite{Thurston_construction}). This classification, now called Nielsen--Thurston classification, provides a powerful and fundamental tool for understanding the dynamics of surface diffeomorphisms. His work \cite{Thurston_norm} also introduced the concept of Thurston norm on the homology of 3-manifolds, which has deep connections to surface theory and mapping class groups.

Formally speaking, given a surface, the mapping class group\index{mapping class group} is defined as the group of orientation preserving homeomorphisms, modulo the subgroup consisting of homeomorphisms isotopic to the identity. 
 For a general background on the subject, see the books of Farb and Margalit \cite{FM_primer} and of Minsky \cite{Minsky_book}.

For one thing, the quotient of the Teichm\"uller space\index{Teichm\"uller space} by the mapping class group is the moduli space of algebraic curves. More precisely, let $S_g$ be the closed orientable connected surface of genus $g \geq 2$. The Teichm\"uller space $\mathcal{T}(S_g)$ is the space of all marked hyperbolic structures on $S_g$. Then the mapping class group $\Mod(S_g)$ acts properly discontinuously on $\mathcal{T}(S_g)$ by isometries and the quotient is the moduli space\index{moduli space} $\mathcal{M}_g$ of algebraic curves of genus $g$. In fact, since $\mathcal{T}(S_g)$ is simply connected, the Teichm\"uller space $\mathcal{T}(S_g)$ is the orbifold universal cover of the moduli space $\mathcal{M}_g$ and $\Mod(S_g)$ is the orbifold fundamental group of $\mathcal{M}_g$.

Given that, understanding various subgroups of $\Mod(S_g)$ is related to understanding various covers of the moduli space. In particular, it would be interesting to study what normal subgroups of  $\Mod(S_g)$ can exist, in order to understand regular covers of $\mathcal{M}_g$. 

One of the most famous examples of proper normal subgroups of the mapping class group is the so-called Torelli group $\mathcal{I}_g$. The Torelli group\index{Torelli group} is defined as the subgroup of $\Mod(S_g)$ of elements which act trivially on $H_1(S_g) := H_1(S_g ;  \mathbb{Z})$. It is easy to see that Dehn twists along separating curves are elements of this group (see Example \ref{ex.Dehn} for the definition of Dehn twist). Another well-known type of elements in the Torelli group is the so-called bounding pair map\index{bounding pair!bounding pair map}. It is of the form $T_{b_1} \circ T_{b_2}^{-1}$ where $b_1, b_2$ are disjoint homologous simple closed curves and $T_{b_i}$ denotes the Dehn twist along the curve $b_i$, $i = 1, 2$. In fact, the Torelli group is generated by Dehn twists along separating curves and bounding pair maps. Another famous example is the Johnson kernel\index{Johnson kernel} $\mathcal{K}_g$. This is the kernel of the Johnson homomorphism and it is the subgroup of $\mathcal{I}_g$ generated by Dehn twists along separating curves. In fact, there exists a whole filtration of proper normal subgroups including these examples, called Johnson filtration, but we are not going into details about it in this chapter. 

Another direction of research is to find subgroups of certain structures. One of the most notable and foundational examples along this line was given by the work of Koberda \cite{Koberda_RAAG}, which shows that if a finite simplicial graph $\Gamma$ is an induced subgraph of the curve graph of $S_g$ which will be defined later, then the right-angled Artin group $A(\Gamma)$ is a subgroup of $\Mod(S_g)$. Later, Clay, Mangahas, and Margalit \cite{CMM_RAAG} showed that there are also normal subgroups of $\Mod(S_g)$ isomorphic to the right-angled Artin groups. See also \cite{KK_RAAG} for related results.

One can ask a slightly different question, namely, what elements can a proper normal subgroup of $\Mod(S_g)$ have? In fact, the question we would like to focus on in this chapter is the opposite question: which elements of $\Mod(S_g)$ are never contained in any proper normal subgroups? From the point of view of $\Mod(S_g) = \pi_1(\mathcal{M}_g)$, 
one can interpret this question as asking which closed curves in the moduli space never lift to a closed curve in any regular cover of the moduli space. In general, an element $h$ of a group $G$ is called a \emph{normal generator}\index{normal generator} if its normal closure $\llangle h \rrangle$ is the entire group $G$, i.e., no proper normal subgroup contains the given element $h$. Hence, our aim is to find normal generators of the mapping class groups.

Maher and Tiozzo proved that normal generators of mapping class groups are not generic \cite{MT_RW}, by showing that the normal closure of a random mapping class is a free group.
This suggests that finding a normal generator of $\Mod(S_g)$ is a challenging task. On the other hand, Lickorish \cite{Lickorish_generators} and Mumford \cite{Mumford_normal} showed that the Dehn twist along a non-separating simple closed curve on $S_g$, which is obtained by cutting, rotating, and regluing the surface along the curve, is a normal generator of $\Mod(S_g)$. Many periodic normal generators are also given by Harvey and Korkmaz (\cite{HK_homomorphisms}, \cite{Korkmaz_generating}), Lanier \cite{Lanier_periodic}, McCarthy and Papadopoulos \cite{MP_involutions}, and Yoshihara \cite{yoshihara2016generating}. Indeed, Lanier and Margalit \cite{LM_NG}  proved that every non-trivial periodic mapping class that is not a hyperelliptic involution is a normal generator of $\Mod(S_g)$ when $g \ge 3$.

For pseudo-Anosov mapping classes, surprisingly, it was shown by Lanier and Margalit \cite{LM_NG} that the normal generation of a mapping class turns out to be related to its (asymptotic) translation length\index{asymptotic translation length} on  Teichm\"uller space: for $f \in \Mod(S_g)$, its translation length on $\T(S_g)$ is defined as $$\ell_{\T}(f) := \lim_{n \to \infty} \frac{d_{\T}(o, f^n(o))}{n}$$
where $d_{\T}$ is the Teichm\"uller distance on $\T(S_g)$ and $o$ is an arbitrary point in $\T(S_g)$. 
More precisely, they showed the following criterion for mapping classes to be normal generators: pseudo-Anosov mapping classes with small translation lengths on Teichm\"uller spaces are normal generators. Together with the work \cite{Penner_bounds} of Penner, it follows that pseudo-Anosov normal generators are abundant in the following sense. 

\begin{theorem}[Lanier--Margalit] \label{thm.lmintro}
Let $f \in \Mod(S_g)$ for $g \geq 3$. Then $f$ is a normal generator if one of the followings holds: 
\begin{itemize}
    \item $f$ is of finite order and is not a hyperelliptic involution.
    \item $f$ is pseudo-Anosov and $\ell_{\T}(f) \le \log \sqrt{2}$.
\end{itemize}
\end{theorem} 

Lanier and Margalit proved this theorem by giving a sufficient and necessary condition for a given mapping class to be a normal generator, which we will discuss in Section \ref{sec:LMcriterion}. In our work with Wu \cite{BKW_reducible}, we extended their theorems to certain reducible mapping classes. Before presenting the statement, we first note that if $f \in \Mod(S_g)$ satisfies $\ell_{\T}(f) > 0$, then there exists a subsurface $A \subset S_g$ which is invariant under some power of $f$ and whose restriction to $A$ is pseudo-Anosov.

\begin{theorem}[Baik--Kim--Wu]
Let $f \in \Mod(S_g)$ be a mapping class that preserves a subsurface $A \subseteq S_g$ of genus at least three and suppose that $f|_A$ is pseudo-Anosov. If $\ell_{\T}(f) \le \log \sqrt{2}$, then $f$ is a normal generator.
\end{theorem}

Another metric space on which $\Mod(S_g)$ naturally acts is the curve graph\index{curve graph} of $S_g$. This graph $\C(S_g)$ is defined as the graph whose vertices are isotopy classes of essential simple closed curves on $S_g$ and where two vertices are connected by an edge if they have disjoint representatives. The curve graph was first introduced by Harvey \cite{Harvey_CC}. We equip $\C(S_g)$ with a simplicial metric $d_{\C}$. For $f \in \Mod(S_g)$, its (asymptotic) translation length\index{asymptotic translation length} on $\C(S_g)$ is defined in the same way: $$\ell_{\C}(f) := \lim_{n \to \infty} \frac{d_{\C}(o, f^n(o))}{n}$$ for $o \in \C(S_g)$.

Masur and Minsky showed in \cite{MM_CC} that $\C(S_g)$ is Gromov hyperbolic. Moreover, $\ell_{\C}(f) > 0$ if and only if $f$ is pseudo-Anosov. Hence, from the point of view of Theorem \ref{thm.lmintro}, it is natural to ask whether small $\ell_{\C}(\cdot)$ implies normal generation (cf. \cite[Question 1.2]{BKSW_asymptotic}). While this question is wide open, the following theorem shows how small it should be to make the question have an affirmative answer:

\begin{theorem}[Baik--Kim--Wu]
    For each $g \ge 578$, there exists a pseudo-Anosov $f_g \in \Mod(S_g)$ such that $$f_g \notin \I_g \quad \text{and} \quad \ell_{\C}(f_g) \le \frac{1152}{g - 577}$$
    while $f_g$ is not a normal generator for $\Mod(S_g)$.
\end{theorem}

Indeed, Baik and Shin \cite{BaikShin_Torelli} showed that there exists $c > 1$ such that 
$$\frac{1}{c \cdot g} \le \inf \{ \ell_{\C}(f) : f \in \I_g \text{ is pseudo-Anosov} \} \le \frac{c}{g}$$
for all $g \ge 2$. The condition $f_g \notin \I_g$ in the above theorem says that elements of the Torelli group are not the only obstruction for pseudo-Anosovs with small $\ell_{\C}(\cdot)$ to be normal generators.

We note that while both $\T(S_g)$ and $\C(S_g)$ are metric spaces on which $\Mod(S_g)$ naturally acts, translation lengths measured on them behave quite differently. For instance, Bader recently showed in \cite{bader2025comparing} that for each $g \ge 2$, there exists a sequence of pseudo-Anosovs $f_n \in \Mod(S_g)$ such that
$$\lim_{n \to \infty} \ell_{\T}(f_n) = \infty \quad \text{and} \quad \ell_{\C}(f_n) \le \frac{1}{g-1} \text{ for all } n \ge 1.$$

This chapter is devoted to the study of the relation between normal generators of mapping class groups and translation lengths on Teichm\"uller spaces and curve graphs. In the rest of this chapter, we discuss some key ideas in the proofs of  the theorems introduced above.
In the last section, we also discuss some further questions which are wide open.

\subsection*{Organization} In Section \ref{sec.MCG}, we review structural aspects of mapping class groups. Section \ref{sec:LMcriterion} is devoted Lanier--Margalit's well-suited criterion and its generalization, which is crucial in the study of normal generation of mapping class groups. We discuss the relation between translation lengths of Teichm\"uller spaces and normal generation in Section \ref{sec.smalltrlength}. In Section \ref{sec.curvegraph}, we move on to the discussion of asymptotic translation lengths on curve graphs. We record various further questions in Section \ref{sec.question}.

\subsection*{Acknowledgements} We would like to extend our gratitude to Chenxi Wu for his valuable collaboration with the authors on the primary work discussed in this chapter. Baik was supported by the National Research Foundation of Korea(NRF) grant funded by the Korea government(MSIT) RS-2025-00513595.

\section{Mapping class groups} \label{sec.MCG}

By a surface\index{surface}, we mean a connected oriented surface of finite genus, possibly with finitely many punctures, and we further assume that its Euler characteristic is negative. The major object of this chapter is the mapping class group of a surface, and this section is devoted to a brief overview of mapping class groups. We refer to \cite{FM_primer} and \cite{Minsky_book} for comprehensive references.

\begin{definition}[Mapping class group and pure mapping class group]
    Let $S$ be a surface. The \emph{mapping class group}\index{mapping class group} $\Mod(S)$ of $S$ is defined as the group of isotopy classes of orientation preserving homeomorphisms: $$\Mod(S) := \Homeo^+(S)/\operatorname{Isotopy}.$$
    The \emph{pure mapping class group}\index{pure mapping class group} $\PMod(S)$ is the subgroup of $\Mod(S)$ consisting of elements fixing each puncture:
    $$\PMod(S) := \{f \in \Mod(S) : f \text{ fixes each puncture of } S\}.$$
\end{definition}

We call elements of $\Mod(S)$ and $\PMod(S)$ mapping classes\index{mapping class group!mapping class} and pure mapping classes\index{pure mapping class group!pure mapping class}. Note that if the number of punctures in $S$ does not exceed one (i.e., $S$ is closed or once-punctured), then $\Mod(S) = \PMod(S)$.

\begin{example}[Dehn twists, \cite{Dehn_twist}] \label{ex.Dehn}
    One example of a pure mapping class is a Dehn twist\index{Dehn twist}. Intuitively, a Dehn twist is constructed by cutting the surface along a simple closed curve, rotating the surface along the curve, and then regluing it. To be precise, let $A = \{r e^{i\theta} : r \in [1, 3], \theta \in [0, 2 \pi]\}$ be an annulus in the complex plane equipped with the standard orientation on it. Let $h : A \to A$ be an orientation preserving homeomorphism defined as $$h(r e^{i\theta}) = r e^{i \theta} e^{- i \pi (r-1)}.$$
    See Figure \ref{fig.twistann}.

    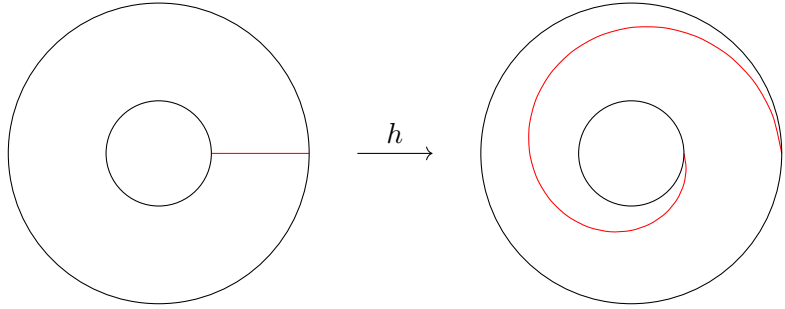
\begin{figure}[h]
    \begin{tikzpicture}[scale=2]

		
		\draw[red,domain=0:1,smooth,variable=\t] plot ({(1-0.65*\t)},{0}); 
		

		\draw (0,0) circle (0.35);
		\draw (0,0) circle (1);
		\end{tikzpicture}
        \begin{tikzpicture}
            \draw (0, -2);
            \draw (0, 2);
            \draw (-1, 0);
            \draw (1, 0);
            \draw[->] (-0.5, 0) -- (0.5, 0);
            \draw (0, 0) node[above] {$h$};
        \end{tikzpicture}
		\begin{tikzpicture}[scale=2]

		\draw[red,domain=0:1,smooth,variable=\t] plot ({(1-0.65*\t)*cos(360*\t)},{(1-0.65*\t)*sin(360*\t)});

		\draw (0,0) circle (0.35);
		\draw (0,0) circle (1);
		\end{tikzpicture}
  \caption{Twist $h$ on an annulus} \label{fig.twistann}

    \end{figure}

    For a simple closed curve $c$ on $S$, the Dehn twist $T_c$ along $c$ is a mapping class defined as follows: let $A_c \subset S$ be a neighborhood of $c$ equipped with an orientation-preserving homeomorphism $f_c : A_c \to A$ that maps $c$ to a circle $\{2 e^{i \theta} : \theta \in [0, 2 \pi] \}$. 
    Then $f_c^{-1} \circ  h \circ f_c : A_c \to A_c$ is an orientation preserving homeomorphism whose restriction on $\partial A_c$ is the identity. We then extend $f_c^{-1} \circ  h \circ f_c$ to the entire surface $S$. The mapping class of this extension is the \emph{Dehn twist} $T_c$ along $c$. See Figure \ref{fig.dehntwistalongc}. As in Theorem \ref{thm.dehnlickorish} below, Dehn twists form a finite generating set of $\PMod(S)$.

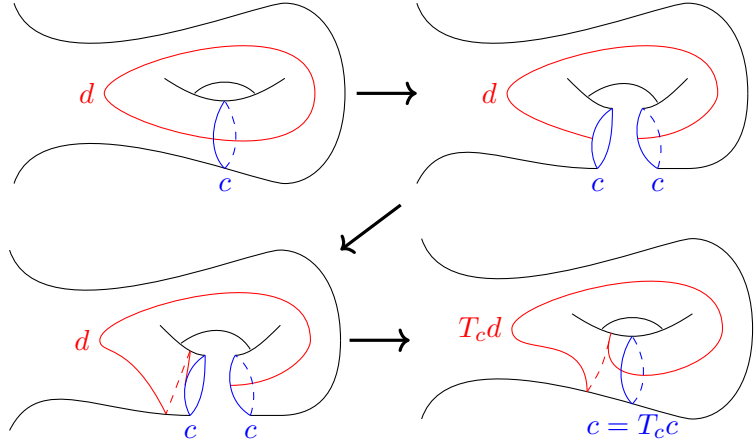
\begin{figure}[h]
	 \begin{tikzpicture}[scale=0.4, every node/.style={scale=1}]

	\draw (-5, 3) .. controls (-4, 0) and (3, 3) .. (4, 3) .. controls (5, 3)  and (6, 2) .. (6, 0) .. controls (6, -2) and (5, -3) .. (4, -3) .. controls (3, -3) and (-4, 0) .. (-5, -3);
	
	\draw (1, 0) .. controls (1.5, 0.5) and (2.5, 0.5) .. (3, 0);
	\draw (0, 0.5) .. controls (1.5, -0.5) and (2.5, -0.5) .. (4, 0.5);
	
	\draw[red] (-2, 0) .. controls (-2, 1) and (5, 3) .. (5, 0) .. controls (5, -3) and (-2, -1) .. (-2, 0);
	\draw[red] (-2, 0) node[left] {$d$};

	\draw[blue] (2, -2.5) .. controls (1.5, -2) and (1.5, -1) .. (2, -0.25);
	\draw[dashed, blue] (2, -2.5) .. controls (2.5, -2) and (2.5, -1) .. (2, -0.25);
	\draw[blue] (2, -2.5) node[below] {$c$};
	\end{tikzpicture} \begin{tikzpicture}[scale=0.4, every node/.style={scale=1}]
	
	\draw[->, very thick] (-7, 0) -- (-5, 0);
	
	\draw (-5, 3) .. controls (-4, 0) and (3, 3) .. (4, 3) .. controls (5, 3)  and (6, 2) .. (6, 0);
	
	\draw (6, 0) .. controls (6, -2) and (5, -2.5) .. (4, -2.5) .. controls (3.5, -2.5) and (3, -2.5) .. (3, -2.5);
	\draw (1, -2.5) .. controls (-2, -2.5) and (-4, -1) .. (-5, -3);
	
	\draw (0.7, -0.1) .. controls (1.5, 0.5) and (2.5, 0.5) .. (3, -0.3);
	
	\draw (0, 0.5) .. controls (0.5, 0) and  (1, -0.5) .. (1.5, -0.5);
	\draw (2.5, -0.5) .. controls (3, -0.5) and (3.5, 0) .. (4, 0.5);

	\draw[red] (0.83, -1.5) .. controls (-1, -1) and (-2, -0.5) .. (-2, 0) .. controls (-2, 1) and (5, 3) .. (5, 0) .. controls (5, -0.5) and (4, -1.5) .. (2.35, -1.5);
	\draw[red] (-2, 0) node[left] {$d$};
	
	\draw[blue] (1.5, -0.5) .. controls (0.7, -1) and (0.7, -2) .. (1, -2.5);
	\draw[blue] (1, -2.5) .. controls (1.5, -2) and (1.5, -1) .. (1.5, -0.5);
	\draw[blue] (1, -2.5) node[below] {$c$};
	
	\draw[blue] (3, -2.5) .. controls (2.2, -2) and (2.2, -1) .. (2.5, -0.5);
	\draw[dashed, blue] (3, -2.5) .. controls (3.2, -2) and (3.2, -1) .. (2.5, -0.5);
	\draw[blue] (3, -2.5) node[below] {$c$};
	
	\end{tikzpicture}
	\begin{tikzpicture}[scale=0.4, every node/.style={scale=1}]
	
	\draw[->, very thick] (8, 4.5) -- (6, 3);

	\draw (-5, 3) .. controls (-4, 0) and (3, 3) .. (4, 3) .. controls (5, 3)  and (6, 2) .. (6, 0);
	
	\draw (6, 0) .. controls (6, -2) and (5, -2.5) .. (4, -2.5) .. controls (3.5, -2.5) and (3, -2.5) .. (3, -2.5);
	\draw (1, -2.5) .. controls (-2, -2.5) and (-4, -1) .. (-5, -3);
	
	\draw (0.7, -0.1) .. controls (1.5, 0.5) and (2.5, 0.5) .. (3, -0.3);
	
	\draw (0, 0.5) .. controls (0.5, 0) and  (1, -0.5) .. (1.5, -0.5);
	\draw (2.5, -0.5) .. controls (3, -0.5) and (3.5, 0) .. (4, 0.5);

	\draw[red] (0.2, -2.45) .. controls (-1, 0) and (-2, -0.5) .. (-2, 0) .. controls (-2, 1) and (5, 3) .. (5, 0) .. controls (5, -0.5) and (4, -1.5) .. (2.35, -1.5);
	
	\draw[red] (1, -0.35) .. controls (1, -0.5) .. (0.83, -1.5);
	
	\draw[red, dashed] (1, -0.35) -- (0.2, -2.45);
	\draw[red] (-2, 0) node[left] {$d$};
	
	\draw[blue] (1.5, -0.5) .. controls (0.7, -1) and (0.7, -2) .. (1, -2.5);
	\draw[blue] (1, -2.5) .. controls (1.5, -2) and (1.5, -1) .. (1.5, -0.5);
	\draw[blue] (1, -2.5) node[below] {$c$};
	
	\draw[blue] (3, -2.5) .. controls (2.2, -2) and (2.2, -1) .. (2.5, -0.5);
	\draw[dashed, blue] (3, -2.5) .. controls (3.2, -2) and (3.2, -1) .. (2.5, -0.5);
	\draw[blue] (3, -2.5) node[below] {$c$};
	
	\draw[->, very thick] (6.3, 0) -- (8.3, 0);
	
	\end{tikzpicture} \begin{tikzpicture}[scale=0.4, every node/.style={scale=1}]

	\draw (-5, 3) .. controls (-4, 0) and (3, 3) .. (4, 3) .. controls (5, 3)  and (6, 2) .. (6, 0) .. controls (6, -2) and (5, -3) .. (4, -3) .. controls (3, -3) and (-4, 0) .. (-5, -3);
	
	\draw (1, 0) .. controls (1.5, 0.5) and (2.5, 0.5) .. (3, 0);
	\draw (0, 0.5) .. controls (1.5, -0.5) and (2.5, -0.5) .. (4, 0.5);

	\draw[red] (0.5, -2.1) .. controls (0.5, 0) and (-2, -1) .. (-2, 0) .. controls (-2, 1) and (5, 3) .. (5, 0) .. controls (5, -1.5) and (0.5, -2.5) .. (1.3, -0.14);
	\draw[red] (-2, 0) node[left] {$T_{c}d$};
	
	\draw[red, dashed] (0.5, -2.1) .. controls (1, -1.3) .. (1.3, -0.14);
	
	\draw[blue] (2, -2.5) .. controls (1.5, -2) and (1.5, -1) .. (2, -0.25);
	\draw[dashed, blue] (2, -2.5) .. controls (2.5, -2) and (2.5, -1) .. (2, -0.25);
	\draw[blue] (2, -2.5) node[below] {$c= T_{c} c$};
	
	\end{tikzpicture}
 \caption{Dehn twist along $c$} \label{fig.dehntwistalongc}
\end{figure}

Note that two isotopic simple closed curves define the same Dehn twist. Hence a Dehn twist can be regarded to be along an isotopy class of a simple closed curve.
A simple closed curve on $S$ is called \emph{essential}\index{simple closed curve!essential} if it is not homotopic to a point or a puncture. One can see that the Dehn twist $T_c$ is non-trivial if and only if $c$ is essential.
    It is easy to see that for $f \in \Mod(S)$, we have $$T_{f(c)} = f T_c f^{-1}.$$
\end{example}

\begin{example}[Multitwists]
    Let $c_1, \cdots, c_k$ be disjoint simple closed curves on $S$. Their union $c := \cup_{i = 1}^k c_i$ is called a multicurve on $S$.
    The \emph{multitwist}\index{multitwist} along the multicurve $c$ is defined as the composition
    $$T_c := T_{c_1} \cdots T_{c_k}.$$
\end{example}
Since each Dehn twist $T_{c_i}$ is supported in a neighborhood of $c_i$, it follows from the disjointness of $c_1, \cdots, c_k$ that $T_{c_1}, \cdots, T_{c_k}$ commute. Hence, the above composition is well-defined, meaning that the product does not depend on the labeling of $c_1, \cdots, c_k$.

Dehn twists are not only the simplest mapping classes, but also the fundamental. It is a classical theorem of Dehn \cite{Dehn_generators} and Lickorish \cite{Lickorish_generators} that $\PMod(S)$ is generated by finitely many Dehn twists. To state the theorem in a more informative way, we introduce some notions. 
An essential simple closed curve is called \emph{non-separating}\index{simple closed curve!non-separating} if its complement is connected, and \emph{separating}\index{simple closed curve!separating} otherwise. For two simple closed curves $c, d$, their \emph{geometric intersection number}\index{intersection number!geometric intersection number} is defined as $$i(c, d) := \inf_{c'\sim c, d'\sim d} \# c \cap d$$
where the infimum is over all simple closed curves $c'$ and $d'$ isotopic to $c$ and $d$ respectively.

\begin{theorem}[Dehn, Lickorish]\index{Dehn--Lickorish theorem} \label{thm.dehnlickorish}
    There exist finitely many non-separating simple closed curves $c_1, \cdots, c_k$ such that $i(c_i, c_j) \le 1$ for all $i, j \in \{1, \cdots k\}$ and Dehn twists $T_{c_1}, \cdots, T_{c_k}$ generate $\PMod(S)$.
\end{theorem}
We note that while Dehn and Lickorish had different approaches, the statement of finite generation by Dehn twists along non-separating simple closed curves is referred to as the Dehn--Lickorish theorem. The precise picture of curves $c_1, \cdots, c_k$ satisfying Theorem \ref{thm.dehnlickorish} can be found in \cite[Fig. 1]{Lickorish_generators}.

Thurston classified mapping classes into three categories \cite{Thurston_construction}.

\begin{theorem}[Thurston's classification]\index{Thurston's classification} \label{thm.ntclass}
    Let $f \in \Mod(S)$. Then one of the following holds:
    \begin{enumerate}
        \item $f$ is \emph{periodic}\index{mapping class!periodic}, i.e., $f$ is of finite order.
        \item $f$ is \emph{reducible}\index{mapping class!reducible}, i.e., $f$ fixes an isotopy class of a multicurve.
        \item $f$ is \emph{pseudo-Anosov}\index{mapping class!pseudo-Anosov}, i.e., there exist a representative $f_0$ of the isotopy class $f$ and a pair of transverse measured foliations that are invariant under $f_0$ and their transverse measures are multiplied by $\la_f$ and $1/\la_f$ respectively for some $\la_f > 1$. In this case, the constant $\la_f$ is called the stretch factor\index{stretch factor} of $f$.
    \end{enumerate}
\end{theorem}

A more detailed discussion on measured foliations and transverse measures is required to precisely understand what pseudo-Anosov means. However, it is enough for us to know the classification theorem and the fact that no power of a pseudo-Anosov mapping class fixes an isotopy class of a simple closed curve, in contrast to reducible mapping classes.

\subsection*{Normal subgroups of $\PMod(S)$}
As mentioned in the introduction, we are interested in normal subgroups.
Given a group $G$, one way to produce its normal subgroup is to consider the commutator subgroup of $G$. For $h, k \in G$, we set $$[h, k] := h k h^{-1}k^{-1} \in G,$$ their commutator\index{commutator}. More generally, for a subset $H \subset G$, we denote $$[H, H]$$ the subgroup of $G$ generated by commutators of elements of $H$. The subgroup $$[G, G] < G$$ is the commutator subgroup\index{commutator subgroup} of $G$. It is easy to see that the commutator subgroup is a normal subgroup. The quotient $G / [G, G]$ is abelian and is called the abelianization\index{abelianization} of $G$.
The group $G$ is called perfect\index{perfect} if $$[G, G] = G.$$
Note that the group $G$ is perfect if and only if $G$ does not admit any surjective homomorphism to a non-trivial abelian group.

Harer proved that $\PMod(S)$ is perfect if the genus of $S$ is at least three \cite{Harer_perfect}. This is a useful tool for checking whether a given (pure) mapping class is a normal generator, as we will see later.

\begin{theorem}[Harer] \label{thm.harer}
    If the genus of $S$ is at least three, then $$[\PMod(S), \PMod(S)] = \PMod(S).$$
\end{theorem}

We also note a well-known fact that $\Mod(S)$ has a trivial center if $S$ is neither a closed surface of genus two nor a torus with one or two punctures. When $S$ is one of those exceptional surfaces, the center of $\Mod(S)$ is isomorphic to $\Z / 2\Z$, consisting of a hyperelliptic involution \cite[Section 3.4]{FM_primer}.

\subsection*{Action on homology groups} We consider the first homology group $H_1(S) := H_1(S; \Z)$\index{homology group}. The mapping class group $\Mod(S)$ has a natural action on $H_1(S)$. We describe how Dehn twists act on $H_1(S)$.

To do this, we introduce the notion of algebraic intersection number\index{intersection number!algebraic intersection number}. For $c, c' \in \pi_1(S)$, their \emph{algebraic intersection number} $\hat{i}(c, c')$ is defined as the sum of the indices of the intersection points of $c$ and $c'$, where the index of an intersection point is $1$ if the orientation of $c$ and $c'$ at the intersection coincides with the orientation of the surface, and $-1$ otherwise. Note that $\hat{i}(c, c')$ depends only on the homology classes of $c$ and $c'$ in $H_1(S)$, hence we use the same notation $\hat{i}( \cdot, \cdot)$ when it concerns homology classes.

Let $c$ be an essential oriented simple closed curve. Then one can observe that the action of the Dehn twist $T_c$ on $H_1(S)$ is as follows: for $x \in H_1(S)$,
\begin{equation} \label{eqn.Dehnhomology}
T_c(x) = x + \hat{i}(c, x)[c].
\end{equation}
Note that while the definition of Dehn twist does not involve the choice of an orientation on a simple closed curve, we let $c$ be oriented in this discussion. This is only for considering the homology class of $[c]$. Indeed, the above expression does not depend on the orientation on $c$.

The $\Mod(S)$-action on the homology group gives another example of a normal subgroup of $\Mod(S)$, which is in fact a normal subgroup of $\PMod(S)$.

\begin{definition}[Torelli group]\index{Torelli group}
    Let $S = S_{g, n}$ be a surface of genus $g$ with $n$  punctures. The \emph{Torelli group} $\I_{g,n }$ is the subgroup of $\PMod(S)$ consisting of elements that act trivially on the first homology $H_1(S)$. 
\end{definition}

In other words, the Torelli group $\I_{g, n}$ is the kernel of the canonical homomorphism $\PMod(S) \to \Aut(H_1(S))$, hence it is a normal subgroup. An element $f \in \PMod(S)$ is called Torelli, or a Torelli element, if $f \in \I_{g, n}$. Note that the Torelli group is non-trivial and a proper subgroup of $\PMod(S)$. Indeed, it follows from  \eqref{eqn.Dehnhomology} that for an essential simple closed curve $c$, the Dehn twist $T_c$ is Torelli if and only if $c$ is separating. Thurston also showed that there are pseudo-Anosov Torelli elements \cite{Thurston_construction}.

\subsection*{Penner's construction of pseudo-Anosovs}\index{Penner's construction}

In fact, the existence of Torelli pseudo-Anosov mapping class is a consequence of Thurston's explicit construction of pseudo-Anosov mapping classes \cite{Thurston_construction}. Thurston's construction of pseudo-Anosov mapping classes was generalized by Penner \cite{Penner_construction}. We close this section by introducing Penner's construction, which is more combinatorial.

Let $S$ be a closed surface of genus $g \ge 2$, for convenience. Let $c_1, \cdots, c_k$ be disjoint simple closed curves on $S$ and $c = \cup_{i = 1}^k c_i$ be a multicurve on $S$. Similarly, let $d_1, \cdots, d_m$ be disjoint simple closed curves on $S$ and set $d = \cup_{i = 1}^m d_i$. We say that $c$ and $d$ fill\index{filling curves} the surface $S$ (or $c_1, \cdots, c_k$ and $d_1, \cdots, d_k$ fill $S$) if every component of $S - (c \cup d)$ is an open disk, when $c$ and $d$ are in minimal position with respect to intersection number. In other words, any essential simple closed curve on $S$ must intersect $c \cup d$. See Figure \ref{fig.penner}.

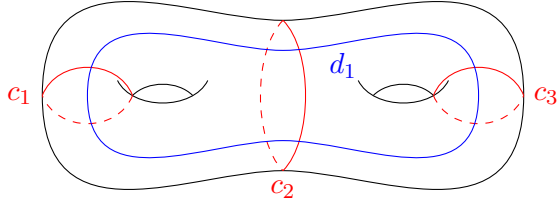
\begin{figure}[h]
	\centering
	\begin{tikzpicture}[scale=2, every node/.style={scale=1}]
	\draw (-1.6, 0) .. controls (-1.6, 1) and (-0.6, 0.5) .. (0, 0.5) .. controls (0.6, 0.5) and (1.6, 1) .. (1.6, 0);
	\begin{scope}[rotate=180]
	\draw (-1.6, 0) .. controls (-1.6, 1) and (-0.6, 0.5) .. (0, 0.5) .. controls (0.6, 0.5) and (1.6, 1) .. (1.6, 0);
	\end{scope}
	
	\draw (-1.1, 0.1) .. controls (-1, -0.1) and (-0.6, -0.1) .. (-0.5, 0.1);
	\draw (-1, 0) .. controls (-0.9, 0.1) and (-0.7, 0.1) .. (-0.6, 0);
	
	\draw (1.1, 0.1) .. controls (1, -0.1) and (0.6, -0.1) .. (0.5, 0.1);
	\draw (1, 0) .. controls (0.9, 0.1) and (0.7, 0.1) .. (0.6, 0);
	
	\draw[red] (1, 0) .. controls (1.1, 0.25) and (1.5, 0.25) .. (1.6, 0);
	\draw[red, dashed] (1, 0) .. controls (1.1, -0.25) and (1.5, -0.25) .. (1.6, 0);
	
	\draw[red] (1.6, 0) node[right] {$c_3$};

 \draw[red] (-1, 0) .. controls (-1.1, 0.25) and (-1.5, 0.25) .. (-1.6, 0);
	\draw[red, dashed] (-1, 0) .. controls (-1.1, -0.25) and (-1.5, -0.25) .. (-1.6, 0);
	
	\draw[red] (-1.6, 0) node[left] {$c_1$};
	
	\draw[red] (0, 0.5) .. controls (0.2, 0.3) and (0.2, -0.3) .. (0, -0.5);
	\draw[red, dashed] (0, 0.5) .. controls (-0.2, 0.3) and (-0.2, -0.3) .. (0, -0.5);
	
	\draw[red] (0, -0.5) node[below] {$c_2$};

        \draw[blue] (-1.3, 0) .. controls (-1.3, 0.7) and (-0.5, 0.3) .. (0, 0.3) .. controls (0.5, 0.3) and (1.3, 0.7) .. (1.3, 0) .. controls (1.3, -0.7) and (0.5, -0.3) .. (0, -0.3) .. controls (-0.5, -0.3) and (-1.3, -0.7) .. (-1.3, 0);
        \draw[blue] (0.4, 0.2) node {$d_1$};

	\end{tikzpicture}
	
	\caption{$c = c_1 \cup c_2 \cup c_3$ and $d = d_1$ fill the surface} \label{fig.penner}
\end{figure}

Penner showed the following criterion for the product of $T_{c_i}$'s and $T_{d_i}$'s to be pseudo-Anosov.

\begin{theorem}[Penner's construction] \label{thm.pennerconstruction}
    Let $c = \cup_{i = 1}^k c_i$ and $d = \cup_{i = 1}^m d_i$ be filling multicurves on $S$. Any product of positive powers of $T_{c_i}$'s and negative powers of $T_{d_i}$'s, where each $c_i$ and each $d_i$ appear at least once, is pseudo-Anosov.
\end{theorem}

Penner conjectured that any pseudo-Anosov mapping class has a power that can be obtained from Penner's construction. Shin and Strenner disproved Penner's conjecture\index{Penner's conjecture} \cite{SS_Penner} using an algebraic approach.

\section{Lanier--Margalit's well-suited criterion and its generalizations} 
\label{sec:LMcriterion}

For a group $G$ and a subset $H \subset G$, we denote by $\llangle H \rrangle$ the normal closure\index{normal closure} of $H$ in $G$. This is the smallest normal subgroup of $G$ containing $H$. An element $h \in G$ is called a normal generator\index{normal generator} of $G$ if its normal closure $\llangle h \rrangle$ is the whole group $G$.
Lanier and Margalit \cite{LM_NG} provided a criterion for a mapping class to be a normal generator of the mapping class group, by considering a certain graph defined for a given mapping class. This criterion is called \emph{well-suited criterion}\index{well-suited criterion}.

Let $S = S_{g, n}$ be a surface of genus $g \ge 1$ with $n \ge 0$ punctures. 
Given a mapping class $f \in \Mod(S)$, we define the graph $N_f(S)$ as follows: 
\begin{itemize}
    \item each vertex is an isotopy class of a non-separating essential simple closed curve;
    \item two vertices $a, b$ are connected by an edge if $(h^{-1}fh)(a) = b$ for some $h \in \Mod(S)$, regarding vertices as isotopy classes of simple closed curves.
\end{itemize}
The graph $N_f(S)$ is called \emph{graph of curves for $f$}\index{graph of curves for $f$}. Intuitively, if the normal closure of $f$ is large enough, then given a generic pair of non-separating essential simple closed curves, one can be mapped to the other by a sequence of conjugates of $f$. This means that the normal generation of $f \in \Mod(S)$ can be encoded by the connectivity of $N_f(S)$. Indeed, we have the following, which is a slight generalization of Lanier--Margalit's well-suited criterion.

\begin{theorem}[Well-suited criterion]\index{well-suited criterion}  \label{thm.wellsuited}
Let $f \in \Mod(S)$. 
If $N_f(S)$ is connected, then $\llangle f \rrangle$ contains the commutator subgroup $[\PMod(S), \PMod(S)]$.
Moreover, when $g \ge 3$, $N_f(S)$ is connected if and only if $$\PMod(S) \le \llangle f \rrangle.$$
\end{theorem} 

Since this is a neat criterion with many useful corollaries, we would like to include its proof here. The proof is organized slightly differently from \cite{LM_NG}.

\begin{proof} 
By Dehn--Lickorish's theorem (Theorem \ref{thm.dehnlickorish}), there exist finitely may non-separating simple closed curves $c_1, \cdots, c_k$ on $S$ such that
\begin{enumerate}
    \item for any $i, j \in \{1, \cdots, k\}$, the geometric intersection number $i(c_i, c_j)$ is either $0$ or $1$;
    \item Dehn twists $T_{c_1}, \cdots, T_{c_k}$ generate $\PMod(S)$.
\end{enumerate}

\medskip
{\bf Step 1.} We claim that the commutator subgroup $[\PMod(S), \PMod(S)]$ is contained in the normal closure of $$\{ T_{c_i} T_{c_j}^{-1} : i, j \in \{ 1, \cdots, k\} \}.$$
For each $m \in \N$, let $P_m$ be the set of pure mapping classes that can be written as products of at most $m$ number of elements in $\{T_{c_1}^{\pm 1}, \cdots, T_{c_k}^{\pm 1}\}$. We then have $$\PMod(S) = \bigcup_{m \in \N} P_m,$$
hence $$[\PMod(S), \PMod(S)] = \bigcup_{m \in \N} [P_m, P_m].$$
On the other hand, for $s, h, l \in \PMod(S)$, we observe that
$$[sh, l] = shlh^{-1}s^{-1}l^{-1} = s(hlh^{-1}l^{-1})s^{-1} (s l s^{-1}l^{-1}) = (s [h, l] s^{-1}) [s, l]$$
and similarly $$[s, hl] = s h l s^{-1}l^{-1}h^{-1} = (sh s^{-1}h^{-1}) h (sl s^{-1} l^{-1}) h^{-1} = [s, h] (h [s, l] h^{-1}).$$
This implies that for each $m \ge 2$, $$[P_m, P_m] \le \llangle [P_{m-1}, P_{m-1}] \rrangle.$$
By an inductive argument, we conclude that $$[\PMod(S), \PMod(S)] \le \llangle [P_1, P_1] \rrangle.$$
Moreover, for any $T_{c_i}$ and $T_{c_j}$, the commutator $[T_{c_i}^{\pm 1}, T_{c_j}^{\pm 1}]$ is contained in the kernel of $\Mod(S) \to \Mod(S) / \llangle T_{c_i} T_{c_j}^{-1} \rrangle$, which is $\llangle T_{c_i} T_{c_j} \rrangle$. Hence, the normal closure of the set $\{ T_{c_i} T_{c_j}^{-1} : i, j \in \{ 1, \cdots, k \}\}$ contains all the commutators $[T_{c_i}^{\pm 1}, T_{c_j}^{\pm 1} ]$, and therefore contains $[P_1, P_1]$. This shows the claim.

\medskip
{\bf Step 2.} We show that for any non-separating curves $c, d$ with $i(c, d) = 1$, $$[\PMod(S), \PMod(S)] \le \llangle T_{c} T_{d}^{-1} \rrangle.$$
Fix two curves $c_i, c_j$. Recall that $i(c_i, c_j)$ is either $0$ or $1$. If $i(c_i, c_j) = 0$, then $[T_{c_i}, T_{c_j}] = e$. Otherwise, if $i(c_i, c_j) = 1$, then it follows from $i(c, d) = 1$ that there exists $h \in \PMod(S)$ such that $h(c) = c_i$ and $h(d) = c_j$. We then have $$T_{c_i} T_{c_j}^{-1} = (h T_c h^{-1}) (h T_{d} h^{-1})^{-1} = h (T_c T_d^{-1})h^{-1}.$$
Therefore, $T_{c_i}T_{c_j}^{-1} \in \llangle T_cT_d^{-1} \rrangle$. By Step 1, the claim follows.

\medskip
{\bf Step 3.} Let $c$ be a non-separating curve such that $i(c, f(c)) = 1$. Then $$[\PMod(S), \PMod(S)] \le \llangle f \rrangle.$$
Indeed, we have $[\PMod(S), \PMod(S)] \le \llangle T_c T_{f(c)}^{-1} \rrangle$ by Step 2. Since $T_c T_{f(c)}^{-1} = (T_c f T_c^{-1}) f^{-1} \in \llangle f \rrangle$, the claim follows. 

\medskip
{\bf Step 4.} Now suppose that $N_f(S)$ is connected. Let $c, d$ be non-separating curves with $i(c, d) = 1$. By the connectivity of $N_f(S)$, there exist conjugates $f_1, \cdots, f_m$ of $f^{\pm 1}$ such that $$d = (f_1 \cdots f_m)(c).$$
By Step 3, this implies that $$[\PMod(S), \PMod(S)] \le \llangle f_1 \cdots f_m \rrangle.$$
Since each $f_i$ is a conjugate of $f^{\pm 1}$, we have $$[\PMod(S), \PMod(S)] \le \llangle f \rrangle$$
as desired.

\medskip
{\bf Step 5.} 
For the last statements of the theorem, let us first assume that $g \geq 3$. By Harer's theorem (Theorem \ref{thm.harer}), $\PMod(S)$ is perfect if $g \geq 3$, i.e., $[\PMod(S), \PMod(S)] = \PMod(S)$. Hence, in this case, if $N_f(S)$ is connected, then $\PMod(S) \le \llangle f \rrangle$ by Step 4. Since $\PMod(S)$ acts transitively on the set of non-separating curves, the converse easily follows. 
\end{proof}

\section{Small translation lengths and normal generation} \label{sec.smalltrlength}

Interestingly, whether a given mapping class is a normal generator can be detected by its dynamical behavior on Teichm\"uller space, as first shown by Lanier and Margalit \cite{LM_NG} and extended by our joint work with Wu \cite{BKW_reducible}. This section is devoted to the discussion on how dynamics on Teichm\"uller spaces is related to normal generation of mapping class groups.

\subsection*{Translation lengths on Teichm\"uller space}
Let $S = S_{g, n}$ be a surface of genus $g \ge 1$ with $n \ge 0$ punctures. The \emph{Teichm\"uller space}\index{Teichm\"uller space} $\T(S)$ is defined as the set of marked hyperbolic structures on $S$. More precisely, $$\T(S) := \{f : S \to X :\text{homeomorphsim to a hyperbolic surface } X\}/ \sim$$
where $h_1 : S \to X_1$ and $h_2 : S \to X_2$ are identified if $h_2 \circ h_1^{-1} : X_1 \to X_2$ is isotopic to an isometry. We use the notation $(h, X)$ for the element $h : S \to X$ of the Teichm\"uller space.

The mapping class group $\Mod(S)$ acts on $\T(S)$ by precomposition: for $f \in \Mod(S)$, $$f \cdot (h, X) = (h \circ f^{-1}, X).$$
There exists a  metric $d_{\T}$ on $\T(S)$, called Teichm\"uller metric, which is invariant under the $\Mod(S)$-action. We omit the precise definition of the Teichm\"uller metric since we will not use it, while we focus more on some dynamical properties of the isometric $\Mod(S)$-action.

\begin{definition}[Asymptotic translation length]\index{asymptotic translation length}
    Let $\mathcal{X}$ be a metric space equipped with a metric $d_{\mathcal{X}}$ and let $f : \mathcal{X} \to \mathcal{X}$ an isometry. The \emph{asymptotic translation length} of $f$ on $\mathcal{X}$ is defined by $$\ell_{\mathcal{X}}(f) := \lim_{m \to \infty} \frac{d_{\mathcal{X}}(o, f^m( o))}{m}$$
    for any $o \in \mathcal{X}$.
\end{definition}

In this section, we consider asymptotic translation lengths of mapping classes on the Teichm\"uller space, equipped with the Teichm\"uller metric. For $f \in \Mod(S)$, we simply write $\ell_{\T}(f) := \ell_{\T(S)}(f)$. 

\subsection*{Pseudo-Anosov normal generators}\index{normal generator!pseudo-Anosov normal generator} Lanier and Margalit showed that pseudo-Anosovs with small asymptotic translation lengths on Teichm\"uller spaces are normal generators. They exist due to Penner \cite{Penner_bounds} (see also Theorem \ref{thm.penner}).  More precisely, they showed the following. We note that the upper bound $\frac{1}{2} \log 2$ does not depend on the surface $S$.

\begin{theorem}[Lanier--Margalit's criterion]\index{Lanier--Margalit's criterion} \cite[Theorem 1.2]{LM_NG} \label{thm.LMtrans}
    Let $f \in \Mod(S)$ be pseudo-Anosov. If $\ell_{\T}(f) \le \frac{1}{2} \log 2$, then $$[\PMod(S), \PMod(S)] \le \llangle f \rrangle.$$
    In particular, if $S$ is closed and of genus at least three, then $$\llangle f \rrangle = \Mod(S).$$
    
\end{theorem}

In fact, the original statement of Lanier and Margalit was about closed surfaces. However, the same proof works for surfaces with finitely many punctures, as the well-suited criterion (Theorem \ref{thm.wellsuited}), which is a key ingredient, allows punctures. See \cite[Section 3]{LM_NG} for the related remark.

The pseudo-Anosov hypothesis on $f$ is crucial in the proof of Lanier--Margalit, as they investigated the combinatorics of a systole with respect to the singular Euclidean metric on the surface induced by a pseudo-Anosov, using \cite[Lemma 2.5, Proposition 2.7]{FLM_lower} which also work for the punctured cases. Since the asymptotic translation length of a pseudo-Anosov mapping class is logarithmic in terms of its stretch factor, as shown in Bers' proof of Thurston's classification theorem \cite{Bers_extremal}, the translation length is captured by the singular Euclidean structure. Using this, Lanier--Margalit showed that $\ell_{\T}(f) \le \frac{1}{2} \log 2$ implies the existence of an essential simple closed curve $c$ such that $i(c, f(c)) \le 2$, which was the key step in the proof of Theorem \ref{thm.LMtrans}, providing a room to apply the well-suited criterion (Theorem \ref{thm.wellsuited}).

\begin{theorem}[Bers]
    Let $f \in \Mod(S)$ be pseudo-Anosov. Then $$\ell_{\T}(f) = \log \la_f.$$
\end{theorem}

Theorem \ref{thm.LMtrans} indeed implies the abundance of pseudo-Anosov normal generators, since there are pseudo-Anosov mapping classes whose asymptotic translation lengths are arbitrary close to $0$. More precisely, for each $g \ge 2$, let
\begin{equation} \label{eqn.mintransteich}
L_{\T}(g) := \inf \{ \ell_{\T}(f) : f \in \Mod(S_g) \text{ is pseudo-Anosov}\}\index{asymptotic translation length!minimal asymptotic translation length}
\end{equation}
where $S_g$ is a closed surface of genus $g$. Penner \cite{Penner_bounds} showed the following asymptotic behavior of $L_{\T}(g)$:

\begin{theorem}[Penner] \label{thm.penner}
    There exists $C > 1$ such that $$\frac{1}{C \cdot g} \le L_{\T}(g) \le \frac{C}{g}$$
    for all $g \ge 2$.
\end{theorem}

Therefore, for any closed surface of sufficiently large genus, we can always find a pseudo-Anosov mapping class whose asymptotic translation length on the Teichm\"uller space is smaller than $\frac{1}{2} \log 2$. By Theorem \ref{thm.LMtrans}, this mapping class turns out to be a normal generator of the mapping class group.

\subsection*{Reducible normal generators}\index{normal generator!reducible normal generator}

Recall from Thurston's classification (Theorem \ref{thm.ntclass}) that there are two types of infinite-order mapping classes, pseudo-Anosovs and reducibles. In our joint work with Wu \cite{BKW_reducible}, we extend Theorem \ref{thm.LMtrans} to deal with reducible mapping classes.

We are especially interested in reducible mapping classes with non-trivial dynamics.
Let $f \in \Mod(S)$ be reducible. Then there are finitely many disjoint simple closed curves $c_1, \cdots, c_k$ whose union $\cup_{i = 1}^k c_i$ is invariant under $f$. This implies that for some $m \in \N$, $f^m$ fixes each connected component of $S - \cup_{i = 1}^k c_i$. Taking the collection $\{c_1, \cdots, c_k\}$ to be maximal,  the restriction of $f^m$ to each component of $S - \cup_{i = 1}^k c_i$ is either periodic or pseudo-Anosov. If every such restriction is periodic, then we can enlarge $m$ so that $f$ is a composition of Dehn twists $T_{c_i}$'s. It then follows that $\ell_{\T}(f) = 0$, which makes considering the asymptotic translation length meaningless. In this regard, we come up with the following notion:

\begin{definition}[Partly pseudo-Anosov mapping class]
    We call $f \in \Mod(S)$ \emph{partly pseudo-Anosov}\index{mapping class!partly pseudo-Anosov} if there exist a representative $f_0$ of $f$ and an embedded subsurface $A \subset S$ so that the isotopy class of the restriction $f_0|_A$ is a pseudo-Anosov element of $\Mod(A)$. We simply denote by $f|_A$ the mapping class of $A$ represented by $f_0|_A$.
\end{definition}

    Some terminologies similar to ``partly pseudo-Anosov" were introduced by several authors. For instance, in \cite{minsky2021bottlenecks}, a mapping class is called ``partial pseudo-Anosov" if its restriction to some subsurface is pseudo-Anosov while it is the identity outside of the subsurface. This is a more restrictive notion than partly pseudo-Anosov mapping class. There is also a terminology ``pure" mapping class, different from the pure mapping class that we use in this chapter, which has a similar feature to partly pseudo-Anosovs (see e.g. \cite{biringer2023homoclinic}).

We now state the generalizatin of Lanier--Margalit's theorem (Theorem \ref{thm.LMtrans}):

\begin{theorem}[Baik--Kim--Wu] \label{thm.bkw}
    Let $S$ be a closed surface. Let $f \in \Mod(S)$ be partly pseudo-Anosov with an invariant subsurface $A$ on which $f|_A$ is pseudo-Anosov. If $\ell_{\T}(f) \le \frac{1}{2} \log 2$ and $A$ has genus at least three, then $$\llangle f \rrangle = \Mod(S).$$
\end{theorem}

\begin{remark}
    In fact, Theorem \ref{thm.locality} still holds when $A$ has genus $\ge 1$ and $S$ is of genus at least three. See \cite[Theorem 3.1]{BKW_reducible}.
\end{remark}

In the rest of this section, we describe the proof of Theorem \ref{thm.bkw}, following \cite{BKW_reducible}. From now on, suppose that $S$ is a closed surface. The key observation is that normal generation of a mapping class can be detected by its local behavior:

\begin{theorem}[Locality of normal generation] \label{thm.locality}
    Let $f \in \Mod(S)$. Suppose that there exists a subsurface $A \subset S$ of genus at least three such that $f(A) = A$. If the normal closure of $f|_A$ in $\Mod(A)$ contains $\PMod(A)$, then $$\llangle f \rrangle = \Mod(S).$$
\end{theorem}

To prove Theorem \ref{thm.locality}, we first show the following lemma. Two disjoint non-separating simple closed curves $c, c'$ are said to form a \emph{bounding pair}\index{bounding pair} if $S - (c \cup c')$ is disconnected.

\begin{lemma} \label{lem.sequence}
    Let $c, c'$ be non-separating simple closed curves. Then there exists a sequence of simple closed curves $a_0, \cdots, a_k$ such that \begin{enumerate}
        \item $a_0 = c$ and $a_k = c'$;
        \item for each $i = 1, \cdots, k$,  $a_{i-1}$ and $a_i$ are disjoint;
        \item for each $i = 1, \cdots, k$, $a_i$ is non-separating;
        \item for each $i = 1, \cdots, k$, $a_{i-1}$ and $a_i$ do not form a bounding pair.
    \end{enumerate}
\end{lemma}

\begin{proof}
    Let $c, c'$ be non-separating simple closed curves. By \cite[Lemma 2.1]{MM_CC}, there exists a sequence of simple closed curves $a_0, \cdots, a_k$ such that $a_0 = c$, $a_k = c'$, and $a_{i-1}$ and $a_i$ are disjoint for each $i = 1, \cdots, k$ (see also Theorem \ref{thm.MMCC}). We modify this sequence so that it also satisfies (3) and (4).

    It is easy to see that we can make the sequence $a_0, \cdots, a_k$ satisfy (3) (see for instance \cite[Theorem 4.4]{FM_primer}). Indeed, if $a_i$ is separating for some $i$, then there are two cases:
    \begin{itemize}
        \item if $a_{i-1}$ and $a_{i+1}$ are contained in the same connected component of $S - a_i$, then we can replace $a_i$ with a non-separating simple closed curve in the other component of $S - a_i$.
        \item otherwise, $a_{i-1}$ and $a_{i+1}$ are already disjoint. Hence we can remove $a_i$ from the sequence.
    \end{itemize}

    We now suppose that the sequence $a_0, \cdots, a_k$ satisfies (1), (2), and (3). To make it satisfy (4) as well, consider the case when $a_{i-1}$ and $a_{i}$ form a bounding pair for some $i$. In this case, we can find a non-separating curve $b_i$ in a component of $S - (a_{i-1} \cup a_i)$ such that both pairs $a_{i-1}$ and $b_i$, and  $b_i$ and $a_i$ are not bounding pairs. Then we can insert $b_i$ between $a_{i-1}$ and $a_i$ to get a new sequence. This yields the desired sequence.
\end{proof}

\subsection*{Proof of Theorem \ref{thm.locality}} We are now ready to prove the locality of normal generation. Let $f \in \Mod(S)$ and $A\subset S$ be as given. By a well-suited criterion (Theorem \ref{thm.wellsuited}), it suffices to show that $N_f(S)$ is connected.

Let $c, c'$ be two vertices in $N_f(S)$. In other words, we take two non-separating simple closed curves. By Lemma \ref{lem.sequence}, we may assume that $c, c'$ are disjoint and do not form a bounding pair, to show that there exists a path in $N_f(S)$ between them.

Denoting by $g$ the genus of $S$, we have that $S - (c \cup c')$ is of genus $g - 2$ with four punctures, since $c$ and $c'$ do not form a bounding pair. Since $A$ has genus at least three (in particular, at least two), there exists $h \in \Mod(S)$ such that $h(c)$ and $h(c')$ are contained in $A$. Since $\PMod(A)$ acts transitively on the space of non-separating curves in $A$ and $\PMod(A)$ is contained in the normal closure of $f|_A$ in $\Mod(A)$, there exist $f_1, \cdots, f_n \in \Mod(A)$ such that $$h(c') = (f_1^{-1} f|_A f_1) \circ \cdots \circ (f_n^{-1} f|_A f_n)(h(c)).$$
We then extend the $f_i$'s to $S$ and obtain $$c' =  (h^{-1}f_1^{-1} f f_1 h) \circ \cdots \circ (h^{-1}f_n^{-1} f f_n h)(c).$$
This implies that $c$ and $c'$ are connected by a path in $N_f(S)$. Therefore, $N_f(S)$ is connected, completing the proof.
\qed

\subsection*{Product region theorem for Teichm\"uller space}\index{Minsky's product region theorem} The last ingredient of the proof of Theorem \ref{thm.bkw} is about the structure of the Teichm\"uller space $\T(S)$.
Let $c_1, \cdots, c_k$ be a disjoint essential simple closed curves on $S$ which are not isotopic to each other. Then a marked hyperbolic structure on $S$ is determined by a restricted marked hyperbolic structure on $S - \cup_{i = 1}^k c_i$, lengths of $c_i$'s, and how much it is twisted along $c_i$'s. and for each $c_i$, the length and the amount of twisting along $c_i$.  For each $\sigma \in \T(S)$ and $c_i$, denote by $\ell_{\sigma}(c_i) \in \R_{>0} $ the length of $c_i$ with respect to $\sigma$ and by $\tau_{\sigma}(c_i) \in \R$ the angle that $\sigma$ is the angle twisted along $c_i$. Adding more curves if necessary, these quantities $\ell_{\sigma}$ and $\tau_{\sigma}$ parametrize $\T(S)$. This parametrization is called Fenchel--Nielsen parametrization\index{Fenchel--Nielsen parametrization}.

We set $\pi(\sigma)$ to be the restriction of $\sigma$ on $S - \cup_{i = 1}^k c_i$. Using the upper half-plane model for the hyperbolic plane\index{hyperbolic plane} $\bH^2 = \{ (x, y) \in \R^2 : y > 0 \}$, we then have a map $$\Pi : \T(S) \to \T(S - \cup_{i = 1}^k c_i) \times \prod_{i = 1}^k \bH_i^2$$
given by $\Pi(\sigma) = (\pi(\sigma), (\tau_{\sigma}(c_1), \ell_{\sigma}(c_1)^{-1}), \cdots, (\tau_{\sigma}(c_k), \ell_{\sigma}(c_k)^{-1}))$, where the $\bH_i^2$'s are copies of $\bH^2$.

We set $X := \T(S - \cup_{i = 1}^k c_i) \times \prod_{i = 1}^k \bH^2$ and equip it with a metric $$d_X := \max \{ d_{\T(S - \cup_{i = 1}^k c_i)}, d_{\bH^2_1}, \cdots, d_{\bH^2_k} \}.$$ Minsky studied how the Teichm\"uller metric $d_{\T}$ on $\T(S)$ and the metric $d_X$ on $X$ are related via the map $\Pi$  \cite{Minsky_product}. He proved the following theorem, asserting that the region of $\T(S)$ on which the length of $\cup_{i = 1}^k c_i$ is short has an almost product structure in a metric sense. The surprising part is that there is only an additive error, not a multiplicative one (compare with a definition of a quasi-isometry):

\begin{theorem}[Minsky's product region theorem, {\cite[Theorem 6.1]{Minsky_product}}]\index{Minsky's product region theorem} \label{thm.minsky} \label{thm.product}
	The above map $\Pi : \T(S) \to X$ is a homeomorphism. Moreover, for $\varepsilon > 0$ sufficiently small, there exists $\delta = \delta(S, \varepsilon)$ such that $$|d_{\T}(\sigma_1, \sigma_2) - d_X(\Pi(\sigma_1), \Pi(\sigma_2))| \le \delta$$ for any $\sigma_1, \sigma_2 \in \T(S)$ such that $\ell_{\sigma_1}(\cup_{i = 1}^k c_i) < \varepsilon$ and $ \ell_{\sigma_2}(\cup_{i = 1}^k  c_i) < \varepsilon$.
\end{theorem}

Using this product region theorem, we finish the proof of Theorem \ref{thm.bkw}.

\subsection*{Proof of Theorem \ref{thm.bkw}} Let $f \in \Mod(S)$ and $A \subset S$ be as given. We consider $\partial A$ as the union of disjoint essential simple closed curves which are not isotopic to each other. Fix $\varepsilon > 0$ small enough so that Theorem \ref{thm.product} applies, and let $\delta > 0$ be the constant given by it.

Let $\sigma \in \T(S)$ be a marked hyperbolic structure such that $\ell_{\sigma}(\partial A) < \varepsilon$. Since $A$ is invariant under $f$, we have $\ell_{f(\sigma)}(\partial A) < \varepsilon$ as well. Hence, it follows from Theorem \ref{thm.product} that $$d_{\T}(\sigma, f^m (\sigma)) \ge d_{\T(A)}(\pi(\sigma), f|_A^m (\pi(\sigma))) - \delta$$  for all $m \in \N$. This implies that $$\ell_{\T}(f|_A) \le \ell_{\T}(f) \le \frac{1}{2} \log 2,$$ therefore  by Theorem \ref{thm.LMtrans}, $\PMod(A)$ is contained in the normal closure of $f|_A$ in $\Mod(A)$. By Theorem \ref{thm.locality}, we conclude that $$\llangle f \rrangle = \Mod(S),$$
showing the normal generation.
\qed

\section{Asymptotic translation lengths on curve graphs} \label{sec.curvegraph}

In the rest of the chapter, let $S = S_g$ be a closed surface of genus $g \ge 2$. There is another metric space on which the mapping class group $\Mod(S)$ acts by isometries, namely, the curve graph. In this section, we discuss the asymptotic translation length on the curve graph and its relation to normal generation.

\subsection*{Curve graph}
The \emph{curve graph}\index{curve graph} $\C(S)$ of $S$ was first introduced by Harvey \cite{Harvey_CC}. It is a graph whose vertices are isotopy classes of essential simple closed curves on $S$, and where two vertices are connected by an edge if they have disjoint representatives. We equip the curve graph with a simplicial metric $d_{\C}$. Some geometric properties of $\C(S)$ were studied by Masur and Minsky \cite{MM_CC}.

\begin{theorem}[Masur--Minsky] \label{thm.MMCC}
    The curve graph  $\C(S)$ is a connected, unbounded, and Gromov hyperbolic metric space.
\end{theorem}
Here, a geodesic metric space is called Gromov hyperbolic\index{Gromov hyperbolicity} if there exists $\delta > 0$ so that every geodesic triangle is $\delta$-thin, i.e., the $\delta$-neighborhood of one side contains the other two sides.

It is clear from the definition that $\Mod(S)$ acts isometrically on $\C(S)$. We use the notation $$\ell_{\C}(\cdot)$$ for the asymptotic translation length on $\C(S)$.
Since there exists a coarsely $\Mod(S)$-equivariant and coarsely Lipschitz map $\T(S) \to \C(S)$ \cite{MM_CC}, there exists $c = c(S) > 0$ such that $\ell_{\C}(f) \le c \cdot \ell_{\T}(f)$ for all $f \in \Mod(S)$. A more effective version of this estimate is obtained in \cite{GHKL_lipschitz}. On the other hand, $\ell_{\C}$ and $\ell_{\T}$ are not close to each other in general \cite{bader2025comparing}.

The Gromov hyperbolicity of $\C(S)$ provides more tools to study the dynamics of $\Mod(S)$ on $\C(S)$.
It was also shown in \cite[Proposition 3.6]{MM_CC} that if $f \in \Mod(S)$ is pseudo-Anosov, then $$\ell_{\C}(f) > 0.$$
It is easy to see that periodic mapping classes and reducible mapping classes have zero asymptotic translation lengths on the curve graph since they have powers fixing a vertex in $\C(S)$. In contrast, note that reducible mapping classes can have positive asymptotic translation lengths on $\T(S)$.

Hence it is natural to ask whether pseudo-Anosovs with small $\ell_{\C}$ are normal generators, in analogy with to Theorem \ref{thm.LMtrans}. This question is wide open, and this section is devoted to introduce some work towards it (cf. \cite[Question 1.2]{BKSW_asymptotic}).

\subsection*{Minimal asymptotic translation lengths}

We first need to clarify what ``small translation lengths" should mean to make the question meaningful. In the case of Teichm\"uller space, a constant $\frac{1}{2} \log 2$ gives a threshold, and notably, is independent of the surface $S=S_g$. This is related to the distribution of asymptotic translation lengths of  Torelli elements\index{Torelli element} on Teichm\"uller spaces. Recall that the Torelli group\index{Torelli group} $\I_g < \Mod(S_g)$ consists of mapping classes acting trivially on $H_1(S_g)$, and is a proper normal subgroup of $\Mod(S_g)$. Hence the threshold must be small enough to exclude Torelli elements.  Farb, Leininger, and Margalit \cite{FLM_lower} showed that, if we denote by 
\begin{equation} \label{eqn.mintrteichtor}
L_{\T}(\I_g) := \inf \{ \ell_{\T}(f) : f \in \I_g \text{ is pseudo-Anosov}\},\index{asymptotic translation length!minimal asymptotic translation length}
\end{equation}
then there exist constants $C_1, C_2 > 0$ such that \begin{equation} \label{eqn.parishotel}
C_1 \le L_{\T}(\I_g) \le C_2
\end{equation} for all $g \ge 2$. This explains why we could expect that pseudo-Anosov mapping classes with $\ell_{\T}$ smaller than a certain universal constant can be a normal generator.

On the other hand, things are very different when we consider asymptotic translation lengths on curve graphs. In analogy with \eqref{eqn.mintransteich} and \eqref{eqn.mintrteichtor}, we define\index{asymptotic translation length!minimal asymptotic translation length}
$$\begin{aligned}
    L_{\C}(g) & := \inf \{ \ell_{\C}(f) : f \in \Mod(S_g) \text{ is pseudo-Anosov}\}; \\
    L_{\C}(\I_g) & := \inf \{ \ell_{\C}(f) : f \in \I_g \text{ is pseudo-Anosov}\}.
\end{aligned}$$
The following asymptotes of $L_{\C}(g)$ and $L_{\C}(\I_g)$ were shown by Gadre and Tsai \cite{GT_minimal} and by Baik and Shin \cite{BaikShin_Torelli} respectively: there exists $C > 1$ such that
\begin{equation} \label{eqn.asymp}
\frac{1}{C \cdot g^2} \le L_{\C}(g) \le \frac{C}{g^2} \quad \text{and} \quad \frac{1}{C \cdot g} \le L_{\C}(\I_g) \le \frac{C}{g} 
\end{equation}
for all $g \ge 2$.

\subsection*{Small translation lengths and normal generation} In view of \eqref{eqn.asymp}, we formulate the following question pertaining to the relation between translation lengths on curve graphs and normal generation as follows:
\begin{question} \label{ques.ng}
Does there exist a constant $C > 0$ such that if $f \in \Mod(S_g)$ is  pseudo-Anosov with $\ell_{\C}(f) \le \frac{C}{g}$, then $f$ is a normal generator?
\end{question}

One can also ask whether the Torelli group is the only obstruction for a pseudo-Anosov with small $\ell_{\C}$ to be a normal generator:

\begin{question} \label{ques.curvegraph}
Does there exist a constant $C > 0$ such that if $f \in \Mod(S_g)$ is a non-Torelli pseudo-Anosov with $\ell_{\C}(f) \le \frac{C}{g}$, then $f$ is a normal generator?
\end{question}

Both questions are still open. However, in our joint work with Wu \cite{BKW_minimal}, we gave an upper bound for $C$ in Question \ref{ques.curvegraph}, if its answer is yes.

\begin{theorem}[Baik--Kim--Wu] \label{thm.curvegraphconj}
    For each $g \ge 578$, there exists a non-Torelli pseudo-Anosov $f_g \in \Mod(S_g)$ such that $$\ell_{\C}(f_g) \le \frac{1152}{g - 577} \quad \text{and} \quad \llangle f_g \rrangle \neq \Mod(S_g).$$
\end{theorem}

The quantities 577 and 578 arise from intersection numbers between curves whose associated Dehn twists are used in the construction of $f_g$.
The rest of this chapter is devoted to explain how we constructed a sequence in the above theorem. An idea to construct pseudo-Anosov mapping classes that do not normally generate mapping class groups is considering finite covers and taking the lifts of a fixed pseudo-Anosov mapping class. Such lifts possess a periodic behavior, from which we deduce that they are not normal generators. Moreover, as the genus gets bigger, more simple closed curves have common disjoint simple closed curves, which would make the distance in the curve graphs smaller. It would be a good exercise to modify the construction to obtain the upper bound of $L_{\C}(\I_g)$ in \eqref{eqn.asymp}. See the work of Baik and Shin \cite{BaikShin_Torelli} for the proof of the precise asymptote of $L_{\C}(\I_g)$.

\subsection*{Construction of coverings}

Let $\alpha$ be a non-separating simple closed curve on the closed surface $S_2$ of genus $2$. Cutting $S_2$ along $\alpha$ and gluing $g$ copies of the resulting surface along copies of $\alpha$ in a cyclic way, we obtain the closed surface $S_{g+1}$ of genus $g+1$. This gives the finite cyclic cover $p_{g+1}$ of degree $g$, as described in Figure \ref{fig:covering}.

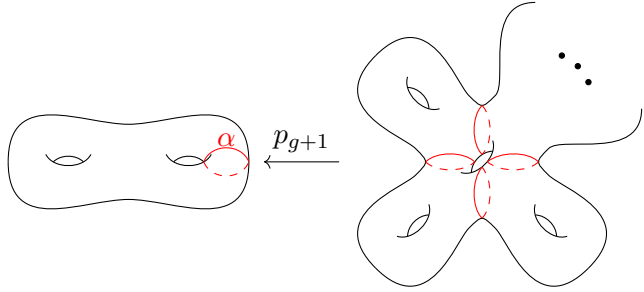
\begin{figure}[h]
	\centering
	\begin{tikzpicture}[scale=1, every node/.style={scale=1}]
	\draw (-1.6, 0) .. controls (-1.6, 1) and (-0.6, 0.5) .. (0, 0.5) .. controls (0.6, 0.5) and (1.6, 1) .. (1.6, 0);
	\begin{scope}[rotate=180]
		\draw (-1.6, 0) .. controls (-1.6, 1) and (-0.6, 0.5) .. (0, 0.5) .. controls (0.6, 0.5) and (1.6, 1) .. (1.6, 0);
	\end{scope}
	
	\draw (-1.1, 0.1) .. controls (-1, -0.1) and (-0.6, -0.1) .. (-0.5, 0.1);
	\draw (-1, 0) .. controls (-0.9, 0.1) and (-0.7, 0.1) .. (-0.6, 0);
	
	\draw (1.1, 0.1) .. controls (1, -0.1) and (0.6, -0.1) .. (0.5, 0.1);
	\draw (1, 0) .. controls (0.9, 0.1) and (0.7, 0.1) .. (0.6, 0);
	
	\draw[red] (1, 0) .. controls (1.1, 0.25) and (1.5, 0.25) .. (1.6, 0);
	\draw[red, dashed] (1, 0) .. controls (1.1, -0.25) and (1.5, -0.25) .. (1.6, 0);
	
	\draw[red] (1.3, 0.3) node {$\alpha$};
	
	\draw[<-] (1.8, 0) -- (2.8, 0);
	\draw (2.3, 0) node[above] {$p_{g+1}$};
	
	\begin{scope}[shift={(4.7, 0)}, rotate=45]
		
		\draw (-2, 0) .. controls (-2, 1) and (-1, 0.5) .. (-0.75, 0.5) .. controls (-0.5, 0.5) .. (-0.5, 0.75) .. controls (-0.5, 1) and (-1, 2) .. (0, 2);
		\draw (-1.5, 0.1) .. controls (-1.4, -0.1) and (-1, -0.1) .. (-0.9, 0.1);
		\draw (-1.4, 0) .. controls (-1.3, 0.1) and (-1.1, 0.1) .. (-1, 0);
		
		\begin{scope}[shift={(1.2, 0)}]
			\draw (-1.5, 0.1) .. controls (-1.4, -0.1) and (-1, -0.1) .. (-0.9, 0.1);
			\draw (-1.4, 0) .. controls (-1.3, 0.1) and (-1.1, 0.1) .. (-1, 0);
		\end{scope}
		
		\begin{scope}[rotate=90]
			\draw (-2, 0) .. controls (-2, 1) and (-1, 0.5) .. (-0.75, 0.5) .. controls (-0.5, 0.5) .. (-0.5, 0.75) .. controls (-0.5, 1) and (-1, 2) .. (0, 2);
			\draw (-1.5, -0.1) .. controls (-1.4, 0.1) and (-1, 0.1) .. (-0.9, -0.1);
			\draw (-1.4, 0) .. controls (-1.3, -0.1) and (-1.1, -0.1) .. (-1, 0);
		\end{scope}
		
		\begin{scope}[rotate=-90]
			\draw (-2, 0) .. controls (-2, 1) and (-1, 0.5) .. (-0.75, 0.5) .. controls (-0.5, 0.5) .. (-0.5, 0.75) .. controls (-0.5, 1) and (-1.5, 1.5) .. (-1, 2);
			\draw (-1.5, 0.1) .. controls (-1.4, -0.1) and (-1, -0.1) .. (-0.9, 0.1);
			\draw (-1.4, 0) .. controls (-1.3, 0.1) and (-1.1, 0.1) .. (-1, 0);
		\end{scope}
		\begin{scope}[rotate=180]
		\draw (-2, 1) .. controls (-1.5, 1.5) and (-1, 0.5) .. (-0.75, 0.5) .. controls (-0.5, 0.5) .. (-0.5, 0.75) .. controls (-0.5, 1) and (-1, 2) .. (0, 2);
		\end{scope}
		
		\draw[red, dashed] (-0.53, 0.53) .. controls (-0.53, 0.33) and (-0.27, 0.07) .. (-0.07, 0.07);
		\draw[red] (-0.53, 0.53) .. controls (-0.33, 0.53) and (-0.07, 0.27) .. (-0.07, 0.07);
		
		\begin{scope}[rotate=-90]
			\draw[red] (-0.53, 0.53) .. controls (-0.53, 0.33) and (-0.27, 0.07) .. (-0.07, 0.07);
			\draw[red, dashed] (-0.53, 0.53) .. controls (-0.33, 0.53) and (-0.07, 0.27) .. (-0.07, 0.07);
		\end{scope}
		
		\begin{scope}[rotate=90]
		\draw[red, dashed] (-0.53, 0.53) .. controls (-0.53, 0.33) and (-0.27, 0.07) .. (-0.05, 0.05);
		\draw[red] (-0.53, 0.53) .. controls (-0.33, 0.53) and (-0.07, 0.27) .. (-0.05, 0.05);
		\end{scope}
		
		\begin{scope}[rotate=180]
		\draw[red] (-0.53, 0.53) .. controls (-0.53, 0.33) and (-0.27, 0.07) .. (-0.05, 0.05);
		\draw[red, dashed] (-0.53, 0.53) .. controls (-0.33, 0.53) and (-0.07, 0.27) .. (-0.05, 0.05);
		\end{scope}
		
		\filldraw (1.75, 0.25) circle(1pt);
		\filldraw (1.8, 0) circle(1pt);
		\filldraw (1.75, -0.25) circle(1pt);
	\end{scope}

	\end{tikzpicture}
	
	\caption{Finite cyclic covering of degree $g$} \label{fig:covering}
\end{figure}

The covering $p_{g+1}$ can also be defined algebraically. Recall the algebraic intersection number $\hat{i}(\cdot, \cdot)$.
Considering the composition $$\pi_{1}(S_2)  \xrightarrow{\hat{i}(\cdot, \alpha)} \Z \xrightarrow{\mod g} \Z / g\Z$$ which is a homomorphism, the covering $p_{g+1}$ corresponds to the kernel of this homomorphism.

\subsection*{Construction of pseudo-Anosovs}

Keep the choice of the simple closed curve $\alpha$. Fixing $g > 1$, we simply denote the covering by $p := p_{g+1}$. We choose a separating simple closed curve $\beta \subset S_2$ as in Figure \ref{fig:alphabeta}.

	\begin{figure}[h]
		\centering
		\begin{tikzpicture}[scale=1, every node/.style={scale=1}]
		\draw (-1.6, 0) .. controls (-1.6, 1) and (-0.6, 0.5) .. (0, 0.5) .. controls (0.6, 0.5) and (1.6, 1) .. (1.6, 0);
		\begin{scope}[rotate=180]
			\draw (-1.6, 0) .. controls (-1.6, 1) and (-0.6, 0.5) .. (0, 0.5) .. controls (0.6, 0.5) and (1.6, 1) .. (1.6, 0);
		\end{scope}
		
		\draw (-1.1, 0.1) .. controls (-1, -0.1) and (-0.6, -0.1) .. (-0.5, 0.1);
		\draw (-1, 0) .. controls (-0.9, 0.1) and (-0.7, 0.1) .. (-0.6, 0);
		
		\draw (1.1, 0.1) .. controls (1, -0.1) and (0.6, -0.1) .. (0.5, 0.1);
		\draw (1, 0) .. controls (0.9, 0.1) and (0.7, 0.1) .. (0.6, 0);
		
		\draw[red] (1, 0) .. controls (1.1, 0.25) and (1.5, 0.25) .. (1.6, 0);
		\draw[red, dashed] (1, 0) .. controls (1.1, -0.25) and (1.5, -0.25) .. (1.6, 0);
		
		\draw[red] (1.3, 0.3) node {$\alpha$};
		
		\draw[blue] (0, 0.5) .. controls (0.2, 0.3) and (0.2, -0.3) .. (0, -0.5);
		\draw[blue, dashed] (0, 0.5) .. controls (-0.2, 0.3) and (-0.2, -0.3) .. (0, -0.5);
		
		\draw[blue] (0, -0.5) node[below] {$\beta$};

		\draw[<-] (1.8, 0) -- (2.8, 0);
		\draw (2.3, 0) node[above] {$p$};
		
		\begin{scope}[shift={(4.7, 0)}, rotate=45]
		
			\draw (-2, 0) .. controls (-2, 1) and (-1, 0.5) .. (-0.75, 0.5) .. controls (-0.5, 0.5) .. (-0.5, 0.75) .. controls (-0.5, 1) and (-1, 2) .. (0, 2);
			\draw (-1.5, 0.1) .. controls (-1.4, -0.1) and (-1, -0.1) .. (-0.9, 0.1);
			\draw (-1.4, 0) .. controls (-1.3, 0.1) and (-1.1, 0.1) .. (-1, 0);
		
			\begin{scope}[shift={(1.2, 0)}]
				\draw (-1.5, 0.1) .. controls (-1.4, -0.1) and (-1, -0.1) .. (-0.9, 0.1);
				\draw (-1.4, 0) .. controls (-1.3, 0.1) and (-1.1, 0.1) .. (-1, 0);
			\end{scope}
		
			\begin{scope}[rotate=90]
				\draw (-2, 0) .. controls (-2, 1) and (-1, 0.5) .. (-0.75, 0.5) .. controls (-0.5, 0.5) .. (-0.5, 0.75) .. controls (-0.5, 1) and (-1, 2) .. (0, 2);
				\draw (-1.5, -0.1) .. controls (-1.4, 0.1) and (-1, 0.1) .. (-0.9, -0.1);
				\draw (-1.4, 0) .. controls (-1.3, -0.1) and (-1.1, -0.1) .. (-1, 0);
			\end{scope}
		
			\begin{scope}[rotate=-90]
				\draw (-2, 0) .. controls (-2, 1) and (-1, 0.5) .. (-0.75, 0.5) .. controls (-0.5, 0.5) .. (-0.5, 0.75) .. controls (-0.5, 1) and (-1.5, 1.5) .. (-1, 2);
				\draw (-1.5, 0.1) .. controls (-1.4, -0.1) and (-1, -0.1) .. (-0.9, 0.1);
				\draw (-1.4, 0) .. controls (-1.3, 0.1) and (-1.1, 0.1) .. (-1, 0);
			\end{scope}
			\begin{scope}[rotate=180]
				\draw (-2, 1) .. controls (-1.5, 1.5) and (-1, 0.5) .. (-0.75, 0.5) .. controls (-0.5, 0.5) .. (-0.5, 0.75) .. controls (-0.5, 1) and (-1, 2) .. (0, 2);
			\end{scope}
			
			\draw[red, dashed] (-0.53, 0.53) .. controls (-0.53, 0.33) and (-0.27, 0.07) .. (-0.07, 0.07);
			\draw[red] (-0.53, 0.53) .. controls (-0.33, 0.53) and (-0.07, 0.27) .. (-0.07, 0.07);
		
			\begin{scope}[rotate=-90]
				\draw[red] (-0.53, 0.53) .. controls (-0.53, 0.33) and (-0.27, 0.07) .. (-0.07, 0.07);
				\draw[red, dashed] (-0.53, 0.53) .. controls (-0.33, 0.53) and (-0.07, 0.27) .. (-0.07, 0.07);
			\end{scope}
		
			\begin{scope}[rotate=90]
				\draw[red, dashed] (-0.53, 0.53) .. controls (-0.53, 0.33) and (-0.27, 0.07) .. (-0.05, 0.05);
				\draw[red] (-0.53, 0.53) .. controls (-0.33, 0.53) and (-0.07, 0.27) .. (-0.05, 0.05);
			\end{scope}
		
			\begin{scope}[rotate=180]
				\draw[red] (-0.53, 0.53) .. controls (-0.53, 0.33) and (-0.27, 0.07) .. (-0.05, 0.05);
				\draw[red, dashed] (-0.53, 0.53) .. controls (-0.33, 0.53) and (-0.07, 0.27) .. (-0.05, 0.05);
			\end{scope}
			
			\draw[blue] (-0.68, 0.5) .. controls (-0.48, 0.3) and (-0.48, -0.3) .. (-0.68, -0.5);
			\draw[blue, dashed] (-0.68, 0.5) .. controls (-0.88, 0.3) and (-0.88, -0.3) .. (-0.68, -0.5);
			
			\begin{scope}[rotate=90]
				\draw[blue] (-0.68, 0.5) .. controls (-0.48, 0.3) and (-0.48, -0.3) .. (-0.68, -0.5);
				\draw[blue, dashed] (-0.68, 0.5) .. controls (-0.88, 0.3) and (-0.88, -0.3) .. (-0.68, -0.5);
			\end{scope}
			
			\begin{scope}[rotate=-90]
			\draw[blue, dashed] (-0.68, 0.5) .. controls (-0.48, 0.3) and (-0.48, -0.3) .. (-0.68, -0.5);
			\draw[blue] (-0.68, 0.5) .. controls (-0.88, 0.3) and (-0.88, -0.3) .. (-0.68, -0.5);
			\end{scope}
		
			\filldraw (1.75, 0.25) circle(1pt);
			\filldraw (1.8, 0) circle(1pt);
			\filldraw (1.75, -0.25) circle(1pt);
		\end{scope}
		
		\end{tikzpicture}
		
		\caption{A separating curve $\beta$ on $S_2$ with $\alpha \cap \beta = \emptyset$} \label{fig:alphabeta}
	\end{figure}
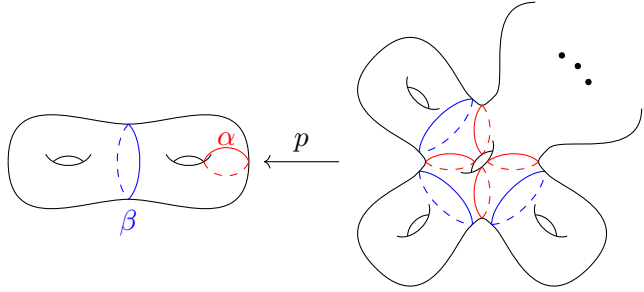

 To construct a pseudo-Anosov, also fix a simple closed curve $\xi$ as in Figure \ref{fig:quant}. Then two simple closed curves $\beta$ and $\xi$ fill the surface $S_2$, and hence $\beta$ and $\la := T_{\xi}\beta$ do so. Hence, by Theorem \ref{thm.pennerconstruction}, the mapping class $\varphi := T_{\la} T_{\beta}^{-1}$ is pseudo-Anosov. Moreover, since $\beta$ is separating, $\la$ is separating as well, and therefore $\varphi$ is Torelli.

 \begin{figure}[h]
	\centering
	\begin{tikzpicture}[scale=1.5, every node/.style={scale=1}]
	\draw (-1.6, 0) .. controls (-1.6, 1) and (-0.6, 0.5) .. (0, 0.5) .. controls (0.6, 0.5) and (1.6, 1) .. (1.6, 0);
	\begin{scope}[rotate=180]
	\draw (-1.6, 0) .. controls (-1.6, 1) and (-0.6, 0.5) .. (0, 0.5) .. controls (0.6, 0.5) and (1.6, 1) .. (1.6, 0);
	\end{scope}
	
	\draw (-1.1, 0.1) .. controls (-1, -0.1) and (-0.6, -0.1) .. (-0.5, 0.1);
	\draw (-1, 0) .. controls (-0.9, 0.1) and (-0.7, 0.1) .. (-0.6, 0);
	
	\draw (1.1, 0.1) .. controls (1, -0.1) and (0.6, -0.1) .. (0.5, 0.1);
	\draw (1, 0) .. controls (0.9, 0.1) and (0.7, 0.1) .. (0.6, 0);
	
	\draw[red] (1, 0) .. controls (1.1, 0.25) and (1.5, 0.25) .. (1.6, 0);
	\draw[red, dashed] (1, 0) .. controls (1.1, -0.25) and (1.5, -0.25) .. (1.6, 0);
	
	\draw[red] (1.6, 0) node[right] {$\alpha$};
	
	\draw[blue] (0, 0.5) .. controls (0.2, 0.3) and (0.2, -0.3) .. (0, -0.5);
	\draw[blue, dashed] (0, 0.5) .. controls (-0.2, 0.3) and (-0.2, -0.3) .. (0, -0.5);
	
	\draw[blue] (0, -0.5) node[below] {$\beta$};
	
	\draw[olive] (-0.6, 0) .. controls (-0.6, 0.4) and (1.2, 0.4) .. (1.2, 0) .. controls (1.2, -0.5) and (-1.4, -0.5) .. (-1.4, 0) .. controls (-1.4, 0.5) and (-0.7, 0.55) .. (-0.6, 0.58);
	\draw[olive, dashed] (-0.6, 0.58) .. controls (-0.3, 0.3) and (0, 0) .. (0.6, 0);
	\draw[olive] (0.6, 0) .. controls (0.6, -0.3) and (-1.2, -0.3) .. (-1.2, 0) .. controls (-1.2, 0.6) and (1.4, 0.6) .. (1.4, 0) .. controls (1.4, -0.5) and (0.7, -0.55) .. (0.6, -0.58);
	\draw[olive, dashed] (0.6, -0.58) .. controls (0.3, -0.5) and (-0.6, -0.2) .. (-0.6, 0);
	
	\draw[olive] (-0.6, 0.58) node[above] {$\xi$};
	\end{tikzpicture}
	
	\caption{$\beta$ and $\xi$ fill the surface} \label{fig:quant}
\end{figure}
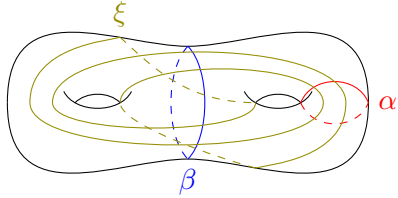

Now we set $f := T_{\beta} T_{\varphi \beta}^{-1} T_{\varphi \alpha}^{-1}$. Since $\alpha$ and $\beta$ are disjoint, $\varphi \alpha$ and $\varphi \beta$ are disjoint as well. Since $\beta$ and $\varphi \beta$ already fill $S_2$, $\beta$ and $\varphi \alpha \cup \varphi \beta$ also fill $S_2$. Hence, by Theorem \ref{thm.pennerconstruction}, $f$ is pseudo-Anosov. Moreover, since $\beta$ and $\varphi \beta$ are separating and $\varphi$ is Torelli, we have the following identities between homology classes in $H_1(S_2)$:
$$[f^{-1}(\alpha)] = [T_{\varphi \alpha} T_{\varphi \beta} T_{\beta}^{-1}\alpha] = [T_{\varphi \alpha} \alpha] = [\varphi T_{\alpha} \varphi^{-1} \alpha] = [\alpha] \in H_1(S_2).$$
This implies that for any $c \in \pi_1(S_2)$, we have $$\hat{i}(f(c), \alpha) = \hat{i}(c, f^{-1}(\alpha)) = \hat{i}(c, \alpha).$$
Therefore, $f$ preserves the kernel of the composition $$\pi_{1}(S_2) \xrightarrow{\hat{i}(\cdot, \alpha)} \Z \xrightarrow{\mod g} \Z / g\Z$$ and hence $f$ has a lift $\tilde f := T_{p^{-1}(\beta)}T_{p^{-1}(\varphi\beta)}^{-1}T_{p^{-1}(\varphi\alpha)}^{-1}$. Again, by Theorem \ref{thm.pennerconstruction}, $\tilde f$ is pseudo-Anosov. We will show that $\tilde f$ is the desired mapping class $f_{g+1}$.

\subsection*{Claim 1: $\tilde f$ is not Torelli} Let us show that $\tilde f$ is not a Torelli element of $\Mod(S_{g+1})$. Let $\eta$ and $\tilde \eta$ be simple closed curves as in Figure \ref{fig:eta}. We then have $[p(\tilde \eta)] = g [\eta] \in H_1(S_2)$. Since $f = T_{\beta} T_{\varphi \beta}^{-1} T_{\varphi \alpha}^{-1}$ and $ T_{\beta} T_{\varphi \beta}^{-1}$ is Torelli, this implies $$[f(\eta)] = [T_{\varphi \alpha}^{-1} \alpha] = [\varphi T_{\alpha}^{-1} \varphi^{-1} \eta].$$
Since $\varphi$ is Torelli as well, we have $$[f(\eta)] = [T_{\alpha}^{-1} \eta] \neq [\eta].$$
This yields $[\tilde f (\tilde \eta)] \neq [\tilde \eta]$, and therefore $\tilde f$ is not Torelli.

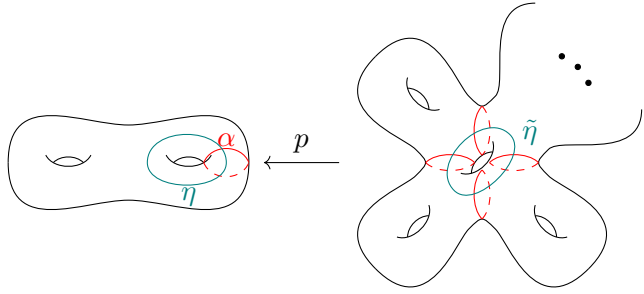
\begin{figure}[h]
		\centering
		\begin{tikzpicture}[scale=1, every node/.style={scale=1}]
			\draw (-1.6, 0) .. controls (-1.6, 1) and (-0.6, 0.5) .. (0, 0.5) .. controls (0.6, 0.5) and (1.6, 1) .. (1.6, 0);
			\begin{scope}[rotate=180]
				\draw (-1.6, 0) .. controls (-1.6, 1) and (-0.6, 0.5) .. (0, 0.5) .. controls (0.6, 0.5) and (1.6, 1) .. (1.6, 0);
			\end{scope}
		
			\draw (-1.1, 0.1) .. controls (-1, -0.1) and (-0.6, -0.1) .. (-0.5, 0.1);
			\draw (-1, 0) .. controls (-0.9, 0.1) and (-0.7, 0.1) .. (-0.6, 0);
		
			\draw (1.1, 0.1) .. controls (1, -0.1) and (0.6, -0.1) .. (0.5, 0.1);
			\draw (1, 0) .. controls (0.9, 0.1) and (0.7, 0.1) .. (0.6, 0);
		
			\draw[red] (1, 0) .. controls (1.1, 0.25) and (1.5, 0.25) .. (1.6, 0);
			\draw[red, dashed] (1, 0) .. controls (1.1, -0.25) and (1.5, -0.25) .. (1.6, 0);

			\draw[red] (1.3, 0.3) node {$\alpha$};

			\draw[teal] (1.3, 0) .. controls (1.3, 0.5) and (0.25, 0.5) .. (0.25, 0) .. controls (0.25, -0.4) and (1.3, -0.4) .. (1.3, 0);
			\draw[teal] (0.8, -0.45) node {$\eta$};

			\draw[<-] (1.8, 0) -- (2.8, 0);
			\draw (2.3, 0) node[above] {$p$};
		
			\begin{scope}[shift={(4.7, 0)}, rotate=45]
		
				\draw (-2, 0) .. controls (-2, 1) and (-1, 0.5) .. (-0.75, 0.5) .. controls (-0.5, 0.5) .. (-0.5, 0.75) .. controls (-0.5, 1) and (-1, 2) .. (0, 2);
				\draw (-1.5, 0.1) .. controls (-1.4, -0.1) and (-1, -0.1) .. (-0.9, 0.1);
				\draw (-1.4, 0) .. controls (-1.3, 0.1) and (-1.1, 0.1) .. (-1, 0);
		
				\begin{scope}[shift={(1.2, 0)}]
					\draw (-1.5, 0.1) .. controls (-1.4, -0.1) and (-1, -0.1) .. (-0.9, 0.1);
					\draw (-1.4, 0) .. controls (-1.3, 0.1) and (-1.1, 0.1) .. (-1, 0);
				\end{scope}
		
				\begin{scope}[rotate=90]
					\draw (-2, 0) .. controls (-2, 1) and (-1, 0.5) .. (-0.75, 0.5) .. controls (-0.5, 0.5) .. (-0.5, 0.75) .. controls (-0.5, 1) and (-1, 2) .. (0, 2);
					\draw (-1.5, -0.1) .. controls (-1.4, 0.1) and (-1, 0.1) .. (-0.9, -0.1);
					\draw (-1.4, 0) .. controls (-1.3, -0.1) and (-1.1, -0.1) .. (-1, 0);

		\draw[red, dashed] (-0.53, 0.53) .. controls (-0.53, 0.33) and (-0.27, 0.07) .. (-0.07, 0.07);
		\draw[red] (-0.53, 0.53) .. controls (-0.33, 0.53) and (-0.07, 0.27) .. (-0.07, 0.07);
				\end{scope}
		
				\begin{scope}[rotate=-90]
					\draw (-2, 0) .. controls (-2, 1) and (-1, 0.5) .. (-0.75, 0.5) .. controls (-0.5, 0.5) .. (-0.5, 0.75) .. controls (-0.5, 1) and (-1.5, 1.5) .. (-1, 2);
					\draw (-1.5, 0.1) .. controls (-1.4, -0.1) and (-1, -0.1) .. (-0.9, 0.1);
					\draw (-1.4, 0) .. controls (-1.3, 0.1) and (-1.1, 0.1) .. (-1, 0);

		\draw[red] (-0.53, 0.53) .. controls (-0.53, 0.33) and (-0.27, 0.07) .. (-0.07, 0.07);
		\draw[red, dashed] (-0.53, 0.53) .. controls (-0.33, 0.53) and (-0.07, 0.27) .. (-0.07, 0.07);
				\end{scope}
				\begin{scope}[rotate=180]
					\draw (-2, 1) .. controls (-1.5, 1.5) and (-1, 0.5) .. (-0.75, 0.5) .. controls (-0.5, 0.5) .. (-0.5, 0.75) .. controls (-0.5, 1) and (-1, 2) .. (0, 2);

		\draw[red] (-0.53, 0.53) .. controls (-0.53, 0.33) and (-0.27, 0.07) .. (-0.07, 0.07);
		\draw[red, dashed] (-0.53, 0.53) .. controls (-0.33, 0.53) and (-0.07, 0.27) .. (-0.07, 0.07);
				\end{scope}
		
				\draw[red, dashed] (-0.53, 0.53) .. controls (-0.53, 0.33) and (-0.27, 0.07) .. (-0.07, 0.07);
				\draw[red] (-0.53, 0.53) .. controls (-0.33, 0.53) and (-0.07, 0.27) .. (-0.07, 0.07);

				\begin{scope}[rotate=-90]

				\end{scope}
				
				\draw[teal] (-1.7+1.175, 0) .. controls (-1.7+1.175, 0.5) and (-0.65+1.175, 0.5) .. (-0.65+1.175, 0) .. controls (-0.65+1.175, -0.4) and (-1.7+1.175, -0.4) .. (-1.7+1.175, 0);
				\draw[teal] (0.565, 0) node[right] {$\tilde{\eta}$};
		
				\filldraw (1.75, 0.25) circle(1pt);
				\filldraw (1.8, 0) circle(1pt);
				\filldraw (1.75, -0.25) circle(1pt);
			\end{scope}
		
		\end{tikzpicture}
		
		\caption{Choice of $\eta$ and $\tilde \eta$} \label{fig:eta}
	\end{figure}

 \subsection*{Claim 2: $\tilde f$ is not a normal generator}
 Since $\varphi$ is Torelli, it also admits a lift $\tilde \varphi$. Recalling that $\tilde f = T_{p^{-1}(\beta)} T_{p^{-1}(\varphi \beta)}^{-1} T_{p^{-1}(\varphi \alpha)}^{-1}$, we have $$\tilde f = T_{p^{-1}(\beta)} \left( \tilde \varphi T_{p^{-1}(\beta)}^{-1} \tilde \varphi^{-1} \right) \left( \tilde \varphi T_{p^{-1}(\alpha)}^{-1} \tilde \varphi^{-1} \right),$$
 which implies $$\llangle \tilde f \rrangle \le \llangle T_{p^{-1}(\beta)},  T_{p^{-1}(\alpha)} \rrangle.$$

 We prove the claim by showing that $\llangle T_{p^{-1}(\beta)},  T_{p^{-1}(\alpha)} \rrangle$ is a proper normal subgroup of $\Mod(S_{g+1})$. Indeed, we will show that $T_{p^{-1}(\beta)}$ and $T_{p^{-1}(\alpha)}$ trivially act on $H_1(S_{g+1}; \Z/g\Z)$. This implies that they belong to the kernel of a canonical homomorphism $\Mod(S_{g+1}) \to \Aut(H_1(S_{g+1}; \Z/g \Z))$.

 As one can see from Figure \ref{fig:alphabeta}, each component of $p^{-1}(\beta)$ is separating. Hence, $T_{p^{-1}(\beta)}$ is Torelli. In particular, $T_{p^{-1}(\beta)}$ acts trivially on $H_1(S_{g+1}; \Z/g\Z)$. In addition, any two components of $p^{-1}(\alpha)$ bound a subsurface,  hence they are homologous. Fixing a component $\tilde \alpha$ of $p^{-1}(\alpha)$, this implies that the action of $T_{p^{-1}(\alpha)}$ on $H_1(S_{g+1})$ is identical to the action of $T_{\tilde \alpha}^{g}$, which acts trivially on $H_1(S_{g+1}; \Z/g\Z)$. This finishes the proof of the claim.

 \subsection*{Asymptotic translation length of $\tilde f$}
 We label the components of $S_{g+1} - p^{-1}(\alpha)$ by $X_1, \cdots, X_g \subset S_{g+1}$ so that $X_i$ and $X_{i+1}$ are glued along one of their boundary components for all $1 \le i \le g$, writing the index $i$ modulo $g$. We keep this convention throughout the section; in particular, $X_0 = X_g$. We make the choice of $\tilde \alpha$ more explicit by setting $\tilde \alpha := \partial X_0 \cap \partial X_1$.

 Since $\varphi$ is Torelli, $\hat{i}(\varphi \alpha, \alpha) = \hat{i}(\varphi \beta, \alpha) = 0$. In particular, both $i(\varphi \alpha, \alpha)$ and $i(\varphi \beta, \alpha)$ are even numbers. Referring to Figure \ref{fig:spreading}, we have \begin{equation}  \label{eqn:trapped}
		T_{p^{-1}(\varphi \alpha)}^{-1}\tilde{\alpha} \subset \bigcup_{j = - i(\varphi \alpha, \alpha)/2}^{i(\varphi \alpha, \alpha)/2} X_j.
	\end{equation}
 Similarly, since $\varphi \alpha$ and $\varphi \beta$ are disjoint, we also have $$T_{p^{-1}(\varphi \beta)}^{-1}T_{p^{-1}(\varphi \alpha)}^{-1}\tilde{\alpha} \subset \bigcup_{j = - {i(\varphi \beta, \alpha) + i(\varphi \alpha, \alpha) \over 2}}^{i(\varphi \beta, \alpha) + i(\varphi \alpha, \alpha) \over 2} X_j.$$ 

 \begin{figure}[h]
		\centering
		\begin{tikzpicture}[scale=1, every node/.style={scale=1}]
		\draw (-1.6, 0) .. controls (-1.6, 1) and (-0.6, 0.5) .. (0, 0.5) .. controls (0.6, 0.5) and (1.6, 1) .. (1.6, 0);
		\begin{scope}[rotate=180]
		\draw (-1.6, 0) .. controls (-1.6, 1) and (-0.6, 0.5) .. (0, 0.5) .. controls (0.6, 0.5) and (1.6, 1) .. (1.6, 0);
		\end{scope}
		
		\draw (-1.1, 0.1) .. controls (-1, -0.1) and (-0.6, -0.1) .. (-0.5, 0.1);
		\draw (-1, 0) .. controls (-0.9, 0.1) and (-0.7, 0.1) .. (-0.6, 0);
		
		\draw (1.1, 0.1) .. controls (1, -0.1) and (0.6, -0.1) .. (0.5, 0.1);
		\draw (1, 0) .. controls (0.9, 0.1) and (0.7, 0.1) .. (0.6, 0);
		
		\draw[red, thick] (1, 0) .. controls (1.1, 0.25) and (1.5, 0.25) .. (1.6, 0);
		\draw[red, dashed, thick] (1, 0) .. controls (1.1, -0.25) and (1.5, -0.25) .. (1.6, 0);
		
		\draw[red] (1.3, 0.3) node {\small $\alpha$};
		
		\draw[olive, thick] (-0.7, -0.05) .. controls (-0.5, -0.4) and (1.2, -0.25).. (1.2, 0.15) .. controls (1.2, 0.3) and (0.6, 0.3) .. (0.6, 0);
		\draw[olive, thick] (-0.6, 0) .. controls (-0.6, 0.4) and (1.5, 0.8) .. (1.5, 0.15) .. controls (1.5, -0.6) and (-0.5, -0.2).. (-0.7, -0.6);
		
		\draw[olive, dashed, thick] (-0.7, -0.6) .. controls (-0.85, -0.4) and (-0.85, -0.25) .. (-0.7, -0.05);
		\draw[olive, dashed, thick] (-0.6, 0) .. controls (-0.4, -0.2) and (0.4, -0.2) .. (0.6, 0);
		
		\draw[<-] (1.8, 0) -- (2.8, 0);
		\draw (2.3, 0) node[above] {$p$};
		
		\begin{scope}[shift={(4.7, 0)}, rotate=45]
		
		\draw (-2, 0) .. controls (-2, 1) and (-1, 0.5) .. (-0.75, 0.5) .. controls (-0.5, 0.5) .. (-0.5, 0.75) .. controls (-0.5, 1) and (-1, 2) .. (0, 2);
		\draw (-1.5, 0.1) .. controls (-1.4, -0.1) and (-1, -0.1) .. (-0.9, 0.1);
		\draw (-1.4, 0) .. controls (-1.3, 0.1) and (-1.1, 0.1) .. (-1, 0);
		
		\begin{scope}[shift={(1.2, 0)}]
		\draw (-1.5, 0.1) .. controls (-1.4, -0.1) and (-1, -0.1) .. (-0.9, 0.1);
		\draw (-1.4, 0) .. controls (-1.3, 0.1) and (-1.1, 0.1) .. (-1, 0);
		\end{scope}
		
		\begin{scope}[rotate=90]
		\draw (-2, 0) .. controls (-2, 1) and (-1, 0.5) .. (-0.75, 0.5) .. controls (-0.5, 0.5) .. (-0.5, 0.75) .. controls (-0.5, 1) and (-1, 2) .. (0, 2);
		\draw (-1.5, -0.1) .. controls (-1.4, 0.1) and (-1, 0.1) .. (-0.9, -0.1);
		\draw (-1.4, 0) .. controls (-1.3, -0.1) and (-1.1, -0.1) .. (-1, 0);
		\end{scope}
		
		\begin{scope}[rotate=-90]
		\draw (-2, 0) .. controls (-2, 1) and (-1, 0.5) .. (-0.75, 0.5) .. controls (-0.5, 0.5) .. (-0.5, 0.75) .. controls (-0.5, 1) and (-1.5, 1.5) .. (-1, 2);
		\draw (-1.5, 0.1) .. controls (-1.4, -0.1) and (-1, -0.1) .. (-0.9, 0.1);
		\draw (-1.4, 0) .. controls (-1.3, 0.1) and (-1.1, 0.1) .. (-1, 0);
		\end{scope}
		\begin{scope}[rotate=180]
		\draw (-2, 1) .. controls (-1.5, 1.5) and (-1, 0.5) .. (-0.75, 0.5) .. controls (-0.5, 0.5) .. (-0.5, 0.75) .. controls (-0.5, 1) and (-1, 2) .. (0, 2);
		\end{scope}
		
		\draw[red, dashed] (-0.53, 0.53) .. controls (-0.53, 0.33) and (-0.27, 0.07) .. (-0.07, 0.07);
		\draw[red] (-0.53, 0.53) .. controls (-0.33, 0.53) and (-0.07, 0.27) .. (-0.07, 0.07);
		
		\begin{scope}[rotate=-90]
		\draw[red] (-0.53, 0.53) .. controls (-0.53, 0.33) and (-0.27, 0.07) .. (-0.07, 0.07);
		\draw[red, dashed] (-0.53, 0.53) .. controls (-0.33, 0.53) and (-0.07, 0.27) .. (-0.07, 0.07);
		\end{scope}
		
		\begin{scope}[rotate=90]
		\draw[red, dashed, thick] (-0.53, 0.53) .. controls (-0.53, 0.33) and (-0.27, 0.07) .. (-0.05, 0.05);
		\draw[red, thick] (-0.53, 0.53) .. controls (-0.33, 0.53) and (-0.07, 0.27) .. (-0.05, 0.05);
		\end{scope}
		
		\begin{scope}[rotate=180]
		\draw[red] (-0.53, 0.53) .. controls (-0.53, 0.33) and (-0.27, 0.07) .. (-0.05, 0.05);
		\draw[red, dashed] (-0.53, 0.53) .. controls (-0.33, 0.53) and (-0.07, 0.27) .. (-0.05, 0.05);
		\end{scope}
		
		\filldraw (1.75, 0.25) circle(1pt);
		\filldraw (1.8, 0) circle(1pt);
		\filldraw (1.75, -0.25) circle(1pt);
		
	
		\draw[olive, dashed, thick] (-0.7-0.5, -0.6) .. controls (-0.85-0.5, -0.4) and (-0.85-0.5, -0.25) .. (-0.7-0.5, -0.05);
		\draw[olive, thick] (-1.2, -0.05) .. controls (-1, -0.3) and (-0.2, -0.3) .. (0, -0.05);
		\draw[olive, thick] (-1.2, -0.6) .. controls (-1, -0.3) and (0, -0.3) .. (0, -1);
		\draw[olive, dashed, thick] (0, -0.05) .. controls (0.1, -0.25) and (0.1, -0.8) .. (0, -1);
		
		\begin{scope}[rotate=180]
		\draw[olive] (-1.2, -0.05) .. controls (-1, -0.3) and (-0.2, -0.3) .. (0, -0.07);
		\draw[olive] (-1.2, -0.6) .. controls (-1, -0.3) and (0, -0.3) .. (0, -1);
		\draw[olive, dashed] (0, -0.07) .. controls (0.1, -0.25) and (0.1, -0.8) .. (0, -1);
  
		\end{scope}
		
		\begin{scope}[rotate=-90]
		\draw[olive, dashed] (-0.7-0.5, -0.6) .. controls (-0.85-0.5, -0.4) and (-0.85-0.5, -0.25) .. (-0.7-0.5, -0.05);
		\draw[olive] (-1.2, -0.05) .. controls (-1, -0.3) and (-0.2, -0.3) .. (0, -0.2);
		\draw[olive] (-1.2, -0.6) .. controls (-1, -0.3) and (0, -0.3) .. (0, -1);
		\draw[olive, dashed] (0, -0.2) .. controls (0.1, -0.25) and (0.1, -0.8) .. (0, -1);

		\end{scope}
		
		\begin{scope}[rotate=90]
		\draw[olive, dashed] (-0.7-0.5, -0.6) .. controls (-0.85-0.5, -0.4) and (-0.85-0.5, -0.25) .. (-0.7-0.5, -0.05);
		\draw[olive] (-1.2, -0.05) .. controls (-1, -0.3) and (-0.2, -0.3) .. (0, -0.2);
		\draw[olive] (-1.2, -0.6) .. controls (-1, -0.3) and (0, -0.3) .. (-0.5, -1);
		\draw[olive, dashed] (0, -0.2) .. controls (-0.1, -0.25) and (-0.1, -0.8) .. (-0.3, -1);

  	\draw[red] (-0.7,0.71) node {$\tilde{\alpha}$};
		\end{scope}
		
		\end{scope}
		
		\end{tikzpicture}
		
		\caption{The geometric intersection number of $\alpha$ and a curve on $S_2$ determines by which subsurfaces the multitwist of $\tilde \alpha$ along the preimage of the curve is trapped as in \eqref{eqn:trapped}.} \label{fig:spreading}
	\end{figure}
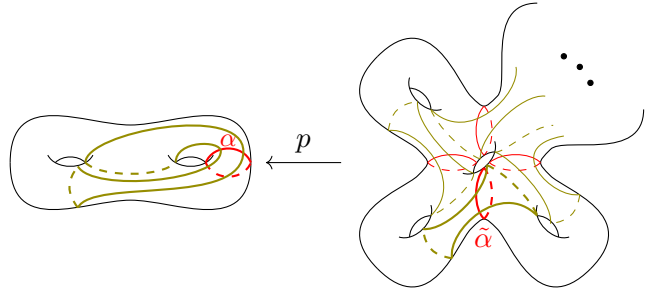

 As one can see from Figure \ref{fig:alphabeta}, $T_{p^{-1}(\beta)}$ fixes each $X_j$. This implies 
	$$
		\tilde{f}\tilde{\alpha} = T_{p^{-1}(\beta)} T_{p^{-1}(\varphi \beta)}^{-1}T_{p^{-1}(\varphi \alpha)}^{-1}\tilde{\alpha} \subset \bigcup_{j = - {i(\varphi \beta, \alpha) + i(\varphi \alpha, \alpha) \over 2}}^{i(\varphi \beta, \alpha) + i(\varphi \alpha, \alpha) \over 2} X_j.
	$$
 Inductively, we have, for any $m \in \N$, $$\tilde{f}^m\tilde{\alpha} \subseteq \bigcup_{j = - m \cdot {i(\varphi \beta, \alpha) + i(\varphi \alpha, \alpha) \over 2}}^{m \cdot {i(\varphi \beta, \alpha) + i(\varphi \alpha, \alpha) \over 2}} X_j.$$

This implies that if $m \in \N$ is such that 
\begin{equation} \label{eqn.estimatem}
m \left(i(\varphi \beta, \alpha) + i(\varphi \alpha, \alpha) \right) + 1 \le g,
\end{equation}
then there exists an essential simple closed curve disjoint from both $\tilde \alpha$ and $\tilde f^m \tilde \alpha$. Indeed, if the inequality in \eqref{eqn.estimatem} is strict, then $X_{\tilde j}$ is disjoint from $\tilde \alpha$ and $\tilde f^m \tilde \alpha$ for some $\tilde j$. If the equality holds in \eqref{eqn.estimatem}, we can take one component of $p^{-1}(\alpha)$. Hence, we have $$d_{\C}(\tilde \alpha, \tilde f^m \tilde \alpha) \le 2.$$
It then follows from the definition of $\ell_{\C}(\cdot)$ that $$\ell_{\C}(\tilde f) = \frac{1}{m} \ell_{\C}(\tilde f^m) \le \frac{2}{m}.$$

Therefore, the estimate for $\ell_{\C}(\tilde f)$ follows once we obtain the largest possible $m$. Recall that $\la = T_{\xi}\beta$ and $\varphi = T_{\la} T_{\beta}^{-1}$. From Figure \ref{fig:quant}, we have $$
\begin{aligned}
i(\xi, \beta) & = 6;\\
i(\lambda, \beta) & = i(T_{\xi}\beta, \beta) = i(\xi, \beta)^2 = 36;
\end{aligned}$$
by \cite[Proposition 3.2]{FM_primer}. It also follows from $\varphi \alpha = T_{\lambda} \alpha$ and $\varphi \beta = T_{\lambda} \beta$ that $$\begin{aligned}
i(\varphi \alpha, \alpha) & = i(T_{\lambda} \alpha, \alpha) = i(\lambda, \alpha)^2  = 144; \\
i(\varphi \beta, \alpha) & = i(T_{\lambda} \beta, \alpha) = i(\lambda, \beta) i(\lambda, \alpha)  = 432.
\end{aligned}$$ Hence, \eqref{eqn.estimatem} becomes $$576 m + 1 \le g.$$
The largest such $m$ also satisfies $$g - 575 \le 576 m + 1 \le g.$$

Consequently, we have shown that if $g \ge 577$, then $$\ell_{\C}(\tilde f) \le \frac{1152}{g - 576}.$$
Since we set $f_{g+1} = \tilde f$, this finishes the proof of Theorem \ref{thm.curvegraphconj}.
\qed

\section{Further questions} \label{sec.question}

We end this chapter by recording some further questions that have not yet been answered. 
Let $S_g$ be a closed surface of genus $g \ge  2$.
We first recall the questions mentioned above:

\begin{question}[Question \ref{ques.ng}] \label{ques.ng2}
Does there exist a constant $C > 0$ such that if $f \in \Mod(S_g)$ is a pseudo-Anosov with $\ell_{\C}(f) \le \frac{C}{g}$, then $f$ is a normal generator?
\end{question}

\begin{question}[Question \ref{ques.curvegraph}] \label{ques.curvegraph2}
Does there exist a constant $C > 0$ such that if $f \in \Mod(S_g)$ is a non-Torelli pseudo-Anosov with $\ell_{\C}(f) \le \frac{C}{g}$, then $f$ is a normal generator?
\end{question}

\subsection*{Minimal translation lengths and normal generation}

As shown by Gadre and Tsai \cite{GT_minimal} \eqref{eqn.asymp}, the following asymptote holds for minimal asymptotic translation lengths on curve graphs:
\begin{equation} \label{eqn.minasympcg}
L_{\C}(g) \asymp \frac{1}{g^2} \quad \text{for all }g \ge 2.
\end{equation}
Hence, as a first step towards Question \ref{ques.ng2}, one can also ask the following weaker version focusing on pseudo-Anosov mapping classes whose asymptotic translation lengths on curve graphs are minimal in the mapping class group.

\begin{question} \label{ques.minimalng}
    Given $g \ge 2$, if $f \in \Mod(S_g)$ satisfies $\ell_{\C}(f) = L_{\C}(g)$, then is $f$ a normal generator?
    Or, is this true for all large enough $g$?
\end{question}

A similar question for asymptotic translation lengths on Teichm\"uller spaces has an affirmative answer by Theorem \ref{thm.LMtrans} (Lanier--Margalit \cite{LM_NG}) and Theorem \ref{thm.penner} (Penner \cite{Penner_bounds}). By Theorem \ref{thm.penner}, we have
\begin{equation} \label{eqn.minasympteich}
L_{\T}(g) \asymp \frac{1}{g} \quad \text{for all } g \ge 2.
\end{equation}
Since Theorem \ref{thm.LMtrans} asserts that a pseudo-Anosov $f \in \Mod(S_g)$ is a normal generator if $\ell_{\T}(f) \le \frac{1}{2} \log 2$ for $g \ge 3$, this answers the $\ell_{\T}$-version of Question \ref{ques.minimalng} affirmatively. 

We also remark that the difference between the two asymptotes \eqref{eqn.minasympcg} and \eqref{eqn.minasympteich} explains the reason for having the genus in the upper bounds for $\ell_\C$ in Question \ref{ques.ng2} and Question \ref{ques.curvegraph2}.

\subsection*{Lanier--Margalit's criterion for general mapping classes}

The original Lanier--Margalit criterion (Theorem \ref{thm.LMtrans}) is about pseudo-Anosov mapping classes. As stated in Theorem \ref{thm.bkw}, it was extended in our joint work with Wu to partly pseudo-Anosov mapping classes, which include certain reducible elements in $\Mod(S_g)$. While this extension still requires mapping classes to have pseudo-Anosov dynamics on some subsurfaces, it is also true that Dehn twists along non-separating essential simple closed curves are also normal generators (\cite{Lickorish_generators}, \cite{Mumford_normal}). Such mapping classes do not exhibit any pseudo-Anosov dynamics on any subsurface, and have zero translation lengths on $\T(S_g)$.
In this regard, we ask for the largest subclass of $\Mod(S_g)$ to which Lanier--Margalit's criterion applies. We first ask whether the criterion applies to all reducible mapping classes.

\begin{question}
    Does there exist $g_0 \in \N$ such that for each $g \ge g_0$, if a reducible $f \in \Mod(S_g)$ satisfies $0 < \ell_{\T}(f) \le \frac{1}{2} \log 2$, then $f$ is a normal generator of $\Mod(S_g)$?
\end{question}

More generally and ambiguously, we ask the following.

\begin{question}
    Is there an alternative characterization for the maximal subset of $\Mod(S_g)$ to which Lanier--Margalit's criterion applies?
\end{question}

\subsection*{Handlebody groups}
While the whole discussion so far has been about the normal generation of the full mapping class group $\Mod(S_g)$. We now change gears and consider a slightly more general question: whether a given subgroup can also be normally generated, or whether there exists a neat description of its normal generators. We mainly discuss a concrete subgroup of $\Mod(S_g)$, the handlebody group, which is an interesting object in itself and has been studied from various points of view.
One can also refer to (\cite{DGO_memoirs}, \cite{CMM_RAAG}) for  related discussions on free subgroups.

Let $V_g$ be a handlebody\index{handlebody} of genus $g \ge 2$, that is, the 3-manifold with boundary obtained by attaching  $g$  1-handles to the 3-ball. We identify $S_g$ with $\partial V_g$ and consider the following subgroup of $\Mod(S_g)$, called the handlebody group\index{handlebody!handlebody group}:
$$
\H_g := \{ f \in \Mod(S_g) : f \text{ extends to } V_g \}.
$$
In other words, the handlebody group $\H_g$ consists of isotopy classes of restrictions of homeomorphisms on $V_g$ to $\partial V_g = S_g$.

The handlebody group $\H_g$ is an infinite, infinite-index subgroup of $\Mod(S_g)$ and is not normal \cite[Corollary 5.4]{hensel2018primer}. Indeed, there are normal generators of $\Mod(S_g)$ in $\H_g$ as we explain now. As in \eqref{eqn.mintransteich}, we consider the following quantity:
\begin{equation} \label{eqn.mintrteichtorhandle}
L_{\T}(\H_g) := \inf \{ \ell_{\T}(f) : f \in \H_g \text{ is pseudo-Anosov}\}.\index{asymptotic translation length!minimal asymptotic translation length}
\end{equation}
Hironaka showed in \cite[Theorem 1.2]{hironaka2011fibered} that 
$$
L_{\T}(\H_g) \asymp \frac{1}{g} \quad \text{for all } g \ge 2.
$$
Therefore, we apply Lanier--Margalit's criterion (Theorem \ref{thm.LMtrans}) and conclude that there are normal generators of $\Mod(S_g)$ in $\H_g$ for all large enough $g$.

On the other hand, it is a different story if we take a normal closure \emph{within} the handlebody group $\H_g$. We first ask whether the handlebody group can be normally generated by a single element:

\begin{question}
    Does there exist $f \in \H_g$ such that the smallest normal subgroup of $\H_g$ containing $f$ is $\H_g$?
    If so, can $f$ be pseudo-Anosov?
\end{question}

Again, we refer to such an element $f \in \H_g$ as a normal generator of $\H_g$.
It might be natural to expect that there are normal generators of $\H_g$ among normal generators of $\Mod(S_g)$ in $\H_g$, which exist as observed above.
From the viewpoint of Lanier--Margalit's criterion, we also ask whether the small asymptotic translation length $\ell_{\T}$ implies a normal generation of $\H_g$.

\begin{question} \label{ques.LMhandle}
    Does there exist $c >  0$ and $g_0 \in \N$ such that for each $g \ge g_0$, if a pseudo-Anosov $f \in \H_g$ satisfies $\ell_{\T}(f) \le c$, then $f$ is a normal generator of $\H_g$?
\end{question}

There are some obstacles to directly adapting the approach of Lanier and Margalit. Note that there are two major steps in their proof, as we also observed in previous sections:
\begin{enumerate}
    \item[Step 1.] As a consequence of the well-suited criterion (Theorem \ref{thm.wellsuited}), if $f \in \Mod(S_g)$ satisfies $\ell_{\T}(f) \le \frac{1}{2} \log 2$, then $$[ \Mod(S_g), \Mod(S_g) ] \le \llangle f \rrangle.$$
    \item[Step 2.] When $g \ge 3$, $\Mod(S_g)$ is perfect (Theorem \ref{thm.harer}), i.e.,
    $$[\Mod(S_g), \Mod(S_g)] = \Mod(S_g).$$
\end{enumerate}

On the other hand, the handlebody group $\H_g$ is not perfect for all $g \ge 2$. Indeed, Wajnryb computed the abelianization of $\H_g$ as follows:

\begin{theorem} \cite[Theorem 20]{Wajnryb_handlebody}
    The abelianization of $\H_g$ is $\Z_2 \oplus \Z_2$ if $g = 2$ and is $\Z_2$ if $g > 2$.
\end{theorem}
In particular, the quotient $\H_g / [\H_g, \H_g]$ is always non-trivial, and hence Step 2 above cannot be adapted to the handlebody group.

Although the well-suited criterion (Theorem \ref{thm.wellsuited}, Step 1 above) does not guarantee the affirmative answer to Question \ref{ques.LMhandle}, we still ask whether one can show the well-suited criterion for the handlebody group.

\begin{question}
    Does the well-suited criterion hold for $\H_g$? More precisely, is there a graph $N_f$ associated with each $f \in \H_g$ so that the connectedness of $N_f$ implies that $[\H_g, \H_g]$ is contained in the normal closure of $f$ in $\H_g$?
\end{question}

It would also be interesting to explore applications of such graphs or the well-suited criterion involving them, beyond addressing the questions mentioned above.

\medskip
\bibliographystyle{alpha} 
\bibliography{NG}

\printindex

\end{document}